\documentclass[a4,11pt]{amsart}
\usepackage{amssymb,amsmath,amsthm,txfonts}
\usepackage{amsrefs}

\newtheorem{thm}{Theorem}[section]
\newtheorem{lem}[thm]{Lemma}

\newtheorem{prop}[thm]{Proposition}
\theoremstyle{definition}
\newtheorem{defn}[thm]{Definition}
\theoremstyle{remark}
\newtheorem{rem}[thm]{\textbf{Remark}}
\newtheorem{rems}[thm]{\textbf{Remarks}}

 \makeatletter
    
    \@addtoreset{equation}{section}
  \makeatother

\makeatletter
      \def\@makefnmark{%
         \leavevmode
            \raise.9ex\hbox{\check@mathfonts
                \fontsize\sf@size\z@\normalfont%
                            \@thefnmark}%
       }
      \makeatother

\newcommand{\cal}{\mathcal}
\newcommand{\p}{\mathbb{P}}
\renewcommand{\q}{\mathbb{Q}}
\newcommand{\R}{\mathbb{R}^{n}}
\newcommand{\D}{\textrm{div}}

\newcommand{\dd}{\textrm{d}}

\begin{document}

\title[]{On estimates for the Stokes flow in a space of bounded functions}
\author[]{Ken Abe}
\date{}
\address[K.ABE]{Department of Mathematics, Faculty of Science, Kyoto University, Kitashirakawa Oiwake-cho, Sakyo-ku, Kyoto 606-8502, Japan}
\email{kabe@math.kyoto-u.ac.jp}
\subjclass[2010]{35Q35, 35K90}
\keywords{Stokes semigroup, Helmholtz projection, composition operator, bounded function spaces, interpolation}
\date{}

\date{}

\maketitle

\begin{abstract}
In this paper, we study regularizing effects of the composition operator $S(t)\mathbb{P}\partial$ for the Stokes semigroup $S(t)$ and the Helmholtz projection $\mathbb{P}$ in a space of bounded functions. We establish new a priori  $L^{\infty}$-estimates of the operator $S(t)\p\partial$ for a certain class of domains including bounded and exterior domains. They imply unique existence of mild solutions of the Navier-Stokes equations in a space of bounded functions.
\end{abstract}

\vspace{10pt}

\section{Introduction and main results}

\vspace{10pt}

We consider the Stokes equations in a domain $\Omega\subset \R$, $n\geq 2$:\\
\begin{align*}
\partial_{t}v-\Delta v+\nabla q&=0\quad \textrm{in}\ \Omega\times (0,T),  \tag{1.1} \\ 
\D\ v&=0\quad \textrm{in}\ \Omega\times (0,T),  \tag{1.2} \\
v&=0\quad \textrm{on}\ \partial\Omega\times (0,T),  \tag{1.3} \\
v&=v_0\quad\hspace{-3.5pt} \textrm{on}\ \Omega\times \{t=0\}.  \tag{1.4} 
\end{align*}
Let $S(t):v_0\longmapsto v(\cdot,t)$ denote the Stokes semigroup and $\p$ denote the Helmholtz projection. In the sequel, $\partial=\partial_{j}$, $j\in \{1,\cdots,n\}$, indiscriminately denotes the spatial derivatives. The goal of this paper is to establish a new a priori $L^{\infty}$-estimate for the composition operator $S(t)\p\partial$. To state a result, let $C^{\infty}_{c}(\Omega)$ denote the space of all smooth functions with compact support in $\Omega$. Let $W^{1,p}(\Omega)$ denote the space of all functions $f\in L^{p}(\Omega)$ such that $\nabla f\in L^{p}(\Omega)$ for $p\in [1,\infty]$. Let $C^{1}_{0}(\Omega)$ denote the closure of $C^{\infty}_{c}(\Omega)$ in $W^{1,\infty}(\Omega)$. One of our main results is the following:    

\vspace{5pt}

\begin{thm}
Let $\Omega$ be a bounded or an exterior domain in $\R$, $n\geq 2$, with $C^{3}$-boundary. For $\alpha\in (0,1)$ and $T_0>0$, there exists a constant $C$ such that 
\begin{equation*}
\big\|S(t)\p\partial f \big\|_{L^{\infty}(\Omega)}\leq 
\frac{C}{t^{\frac{1-\alpha}{2}}}
\big\|f\big\|_{L^{\infty}(\Omega)}^{1-\alpha}
\big\|\nabla f\big\|_{L^{\infty}(\Omega)}^{\alpha}  \tag{1.5}  
\end{equation*}
holds for $f\in C^{1}_{0}\cap W^{1,2}(\Omega)$ and $t\leq T_0$. When $\Omega$ is bounded, (1.5) holds for $T_{0}=\infty$. 
\end{thm}

\vspace{5pt}

The composition operator $S(t)\p\partial$ as well as the Stokes semigroup $S(t)$ plays a fundamental role for studying the nonlinear Navier-Stokes equations. It is well known that $S(t)\p\partial$ acts as a bounded operator on $L^{p}$ ($1<p<\infty$) and satisfies the estimate of the form
\begin{equation*}
\big\|S(t)\p\partial f\big\|_{L^{p}(\Omega)}\leq \frac{C_{p}}{t^{\frac{1}{2}}}\big\| f \big\|_{L^{p}(\Omega)}, \tag{1.6}
\end{equation*}
for $f\in W^{1,p}(\Omega)$ and $t\leq T_0$. Since the Helmholtz projection $\p$ acts as a bounded operator on $L^{p}$, the estimate (1.6)  follows from the analyticity of the Stokes semigroup on $L^{p}$ \cite{Sl77}, \cite{G81}. Recently, analyticity of the Stokes semigroup on $C_{0,\sigma}(\Omega)$ has been proved in \cite{AG1} (\cite{AG2}, \cite{AGH}), where $C_{0,\sigma}(\Omega)$ is the $L^{\infty}$-closure of $C_{c,\sigma}^{\infty}(\Omega)$, the space of all smooth solenoidal vector fields with compact support in $\Omega$. Although the Stokes semigroup is analytic on $C_{0,\sigma}$, the $L^{\infty}$-estimate (1.5) does not follow from the analyticity of the semigroup since the projection $\p$ is not bounded on $L^{\infty}$.

The estimate (1.5) has an application for the Navier-Stokes equations. So far, $L^{\infty}$-type results of the Navier-Stokes equations were established only for the whole space \cite{GIM} (\cite{GMSa}) and a half space \cite{Sl03}, \cite{BJ} for which explicit solution formulas of the Stokes semigroup are available. The difficulties lay on the $L^{\infty}$-estimate of the composition operator $S(t)\mathbb{P}\partial$ as well as the analyticity of the semigroup. Since $C_{0}^{1}$ is the $W^{1,\infty}$-closure of $C^{\infty}_{c}$, the estimate (1.5) yields a unique extension $\overline{S(t)\mathbb{P}\partial}$ acting as a bounded operator from $C^{1}_{0}$ to $C_{0,\sigma}$. (Note that the extension $\overline{S(t)\mathbb{P}\partial}$ is not expressed by the individual operators on $L^{\infty}$). Recently, the estimate (1.5) applies to construct mild solutions of the Navier-Stokes equations on $C_{0,\sigma}$ \cite{A3}.\\

In the sequel, we establish the a priori estimate for 
\begin{equation*}
N(v,q)(x,t)=\bigl|v(x,t)\bigr|+t^{\frac{1}{2}}\bigl|\nabla v(x,t)\bigr|+t\bigl|\nabla^{2}v(x,t)\bigr|+t\bigl|\partial_{t}v(x,t)\bigr|+t\bigl|\nabla q(x,t)\bigr|
\end{equation*} 
of the form
\begin{equation*}
\sup\limits_{0\leq t\leq T_0}t^{\frac{1-\alpha}{2}}\bigl\|N(v,q)\bigr\|_{L^{\infty}(\Omega)}(t)
\leq 
C \Big[f\Big]^{(\alpha)}_{\Omega}   \tag{1.7}
\end{equation*}
for all solutions $(v,q)$ of (1.1)--(1.4) for $v_0=\p \partial f$ with some constants $T_0$ and $C$, where $[f]^{(\alpha)}_{\Omega}$ denotes the H\"older semi-norm of $f$ in $\Omega$, i.e., 
\begin{equation*}
\Big[f\Big]^{(\alpha)}_{\Omega}=\sup\left\{\frac{\big|f(x)-f(y)\big|}{|x-y|^{\alpha}}\ \middle|\ x,y\in \Omega,\ x\neq y \right\}.
\end{equation*}
Since the H\"older semi-norm $[f]^{(\alpha)}_{\Omega}$ is estimated by $||f||_{\infty}^{1-\alpha}||\nabla f||_{\infty}^{\alpha}$ for $f\in C^{\infty}_{c}$, the estimate (1.5) follows from the a priori estimate (1.7). The solutions $(v,q)$ of the Stokes equations (1.1)--(1.4) are given by the Stokes semigroup $S(t)$ and the Helmholtz projection $\p$ on $L^{p}$. We call $(v,q)$ $L^{p}$-solution. We prove Theorem 1.1 from the following: 

\vspace{5pt}

\begin{thm}
Let $\Omega$ be a bounded or an exterior domain with $C^{3}$-boundary. Let $\alpha\in (0,1)$ and $p>n/(1-\alpha)$. For $T_0>0$ there exists a constant $C$ such that (1.7) holds for all $L^{p}$-solutions $(v,q)$ for $v_0=\p \partial f$, $f\in C^{\infty}_{c}(\Omega)$. Moreover, the estimate 
\begin{equation*}
\sup_{0\leq t\leq T_0}t^{\frac{1-\alpha}{2}+s+\frac{|k|}{2}}\big\|\partial_{t}^{s}\partial_{x}^{k}S(t)\p \partial f \big\|_{L^{\infty}(\Omega)}(t)
\leq C\big\|f\big\|^{1-\alpha}_{L^{\infty}(\Omega)}\big\|\nabla f\big\|^{\alpha}_{L^{\infty}(\Omega)}    \tag{1.8}
\end{equation*}
holds for $f\in C^{1}_{0}\cap W^{1,2}(\Omega)$ and $0\leq 2s+|k|\leq 2$.
\end{thm}

\vspace{5pt}

We prove (1.7) by a blow-up argument. It is shown in \cite{AG1} by a blow-up argument that the Stokes semigroup is an analytic semigroup on $C_{0,\sigma}(\Omega)$ for not only bounded domains but also unbounded domains for which some a priori estimate holds for the Neumann problem of the Laplace equation in $\Omega$. We call such a domain \textit{admissible} and it is proved in \cite{AG1} that bounded domains of class $C^{3}$ are admissible. Later, admissibility is proved in \cite{AG2} for exterior domains and in \cite{A1} (\cite[Remarks 1.5 (i)]{AGH}) for a perturbed half space. More recently, admissibility is studied in \cite{AGSS1} for two-dimensional sector-like domains and in \cite{AGSS2} for cylindrical domains. In order to establish (1.7), we introduce a stronger term \textit{strongly admissible}. The term strongly admissible is explained later in the introduction. In this paper, we prove that bounded and exterior domains of class $C^{3}$ are strongly admissible.

We prove (1.7) for general strongly admissible, uniformly $C^{3}$-domains based on the $\tilde{L}^{p}$-theory developed in \cite{FKS1}, \cite{FKS2}, \cite{FKS3}. It is proved in these works that the Helmholtz projection yields a unique decomposition on $\tilde{L}^{p}=L^{p}\cap L^{2}$ ($p\geq 2$) and the Stokes semigroup is analytic on $\tilde{L}^{p}$ for general uniformly $C^{2}$-domains. Thus, solutions of (1.1)--(1.4) exist in a general uniformly $C^{2}$-domain. We prove (1.7) for their $\tilde{L}^{p}$-solutions. The following Theorem 1.3 is a general form of Theorem 1.2.

\vspace{5pt}

\begin{thm}
Let $\Omega$ be a strongly admissible, uniformly $C^{3}$-domain. Let $\alpha\in (0,1)$ and $p>n/(1-\alpha)$. Then, the estimate (1.7) holds for all $\tilde{L}^{p}$-solutions $(v,q)$ for $v_0=\p \partial f$, $f\in C^{\infty}_{c}(\Omega)$. Moreover, (1.8) holds for $f\in C^{1}_{0}\cap W^{1,2}(\Omega)$.
\end{thm}

\vspace{10pt}

Let us sketch the proof of the a priori estimate (1.7). When $\Omega$ is the whole space, the Stokes semigroup agrees with the heat semigroup (i.e., $v=e^{t\Delta}\p \partial f$, $\nabla q\equiv 0$). We estimate $v=\partial e^{t\Delta}\p f$ by the H\"older semi-norm of $\p f$, i.e., 
\begin{equation*}
\big\|\partial e^{t\Delta}\p f\big\|_{\infty}\leq \frac{C}{t^{\frac{1-\alpha}{2}}}\Big[\p f\Big]_{\R}^{(\alpha)}.
\end{equation*}
Since the H\"older semi-norm of $\p f$ is estimated by $[f ]^{(\alpha)}_{\R}$ (see Proposition 3.1), the estimate (1.7) holds for $0<\alpha<1$. (We are able to prove the case $\alpha=0$ by estimating the Oseen kernel $K_{t}$, i.e., $e^{t\Delta}\p f=K_{t}*f$; see \cite{GIM}, \cite{ShibaS} for the whole space and \cite{Sl03}, \cite{BJ} for a half space).

We prove (1.7) by a blow-up argument. For simplicity, we set $\gamma=(1-\alpha)/2$. We prove the existence of constants $T_0$ and $C$ such that (1.7) holds for all $f\in C_{c}^{\infty}(\Omega)$. Suppose on the contrary that (1.7) were false. Then, there would exist a sequence of solutions for (1.1)--(1.4), $(v_m,q_m)$ for $v_{0,m}=\p_{\Omega} \partial f_m$ such that 
\begin{equation*}
\sup_{0\leq t\leq 1/m}t^{\gamma}\big\|N(v_m,q_m)\big\|_{L^{\infty}(\Omega)}(t)>m\Big[f_m\Big]^{(\alpha)}_{\Omega}.
\end{equation*}
We take a point $t_m\in (0,1/m)$ such that 
\begin{equation*}
t^{\gamma}_{m}\bigl\|N(v_m,q_m)\bigr\|_{L^{\infty}(\Omega)}(t_m) \geq \frac{1}{2}M_m, \quad M_m = \sup\limits_{0\leq t\leq 1/m}t^{\gamma}\bigl\|N(v_m,q_m)\bigr\|_{L^{\infty}(\Omega)}(t),
\end{equation*}
and normalize $(v_m,q_m)$ by dividing by $M_m$ to get $\tilde{v}_m=v_m/M_m$, $\tilde{q}_m=q_m/M_m$ and $\tilde{f}_{m}=f_m/M_m$ satisfying 
\begin{alignat*}{2}
\sup\limits_{0\leq t\leq t_m}t^{\gamma}\bigl\|N(\tilde{v}_m,\tilde{q}_m)\bigr\|_{L^{\infty}(\Omega)}(t)&\leq 1,  \\
t^{\gamma}_{m}\bigl\|N(\tilde{v}_m,\tilde{q}_m)\bigr\|_{L^{\infty}(\Omega)}(t_m)&\geq \frac{1}{2},  \\
\Big[\tilde{f}_m\Big]^{(\alpha)}_{\Omega}&<\frac{1}{m}. 
\end{alignat*}
 Then, we rescale $(\tilde{v}_m,\tilde{q}_m)$ around a point $x_m \in \Omega$ such that
\begin{equation*}
t^{\gamma}_{m}N(\tilde{v}_m,\tilde{q}_m)(x_m,t_m) \geq \frac{1}{4} 
\end{equation*}
to get a blow-up sequence $(u_m,p_m)$ of the form
\begin{equation*}
u_m(x,t)=t^{\gamma}_{m}\tilde{v}_m(x_m + t_m^{\frac{1}{2}}x,t_m t),
\quad p_m(x,t)=t_m^{\gamma+\frac{1}{2}}\tilde{q}_m(x_m + t_m^{\frac{1}{2}}x,t_m t),
\end{equation*}
and
\begin{equation*}
 g_m(x)=t_m^{-\frac{\alpha}{2}}\tilde{f}_{m}(x_m+t_m^{\frac{1}{2}}x ).
\end{equation*}
The blow-up sequence $(u_m,p_m)$ satisfies (1.1)--(1.4) for $u_{0,m}=\p_{\Omega_m}\partial g_m$ in $\Omega_m\times (0,1]$ and the rescaled domain $\Omega_m$ expands to either the whole space or a half space as $m\to\infty$. 

The basic strategy is to prove a compactness of the blow-up sequence $(u_m,p_m)$ and a uniqueness of a blow-up limit. If $(u_m,p_m)$ converges to a limit $(u,p)$ strongly enough, one gets a bound from below $N(u,p)(0,1)\geq 1/4$. On the other hand, $(u,p)$ solves a limit problem for $u(\cdot,0)=0$ in a suitable sense. If the limit $(u,p)$ is unique, $u\equiv 0$ and $\nabla p\equiv 0$ follow. This yields a contradiction. For the compactness of $(u_m,p_m)$, we apply the local H\"older estimates for (1.1)--(1.4) proved in \cite{AG1} (to get an equi-continuity of $(u_m,p_m)$). For the uniqueness of $(u,p)$, we develop a uniqueness theorem in a half space. The uniqueness of (1.1)-(1.4) in a half space was proved in \cite{Sl03} for bounded velocity up to time zero. We extend the result for velocity which may not be bounded at time zero based on the previous work of the author \cite{A2}. When $\Omega_m$ expands to the whole space, the uniqueness is reduced to the heat equation.  

A key step of the proof is to get a sufficiently strong initial condition for the blow-up limit $(u,p)$ in order to apply a uniqueness theorem. If the initial data $u_{0,m}$ does not involve the Helmholtz projection $\p_{\Omega_m}$, it is easy to see $u_{0,m}\to 0$ (in a suitable weak sense) as $[g_m]_{\Omega_m}^{(\alpha)}\to 0$ and $m\to\infty$. However, it is non-trivial whether $u_{0,m}=\p_{\Omega_m}\partial g_m\to 0$ as $[g_m]_{\Omega_m}^{(\alpha)}\to 0$ because of the term $\nabla\Phi_{0,m}=\q_{\Omega_m}\partial g_m$ where $\q_{\Omega_m}=I-\p_{\Omega_m}$. When $\Omega$ is the whole space, the projection $\q_{\R}$ has an explicit form by the fundamental solution of the Laplace equation $E$. In fact, $\nabla \Phi_{1}=\q_{\R}\partial f$ agrees with $-\nabla \D h$ for $h=E*\partial f$ so the H\"older semi-norm of $\nabla h$ is estimated by $[f]^{(\alpha)}_{\R}$ and 
\begin{equation*}
\Big[\Phi_{1} \Big]^{(\alpha)}_{\R} \leq C_{\alpha}\Big[f\Big]^{(\alpha)}_{\R}.   \tag{1.9}
\end{equation*}
Since the H\"older estimate (1.9) is scale invariant, it is inherited to $\nabla\Phi_{1,m}=\q_{\R}\partial g_m$. We need a corresponding estimate for $\nabla \Phi_{0}=\q_{\Omega}\partial f$. For this purpose, we consider the Neumann problem of the Laplace equation
\begin{equation*}
\Delta \Phi=0\quad \textrm{in}\ \Omega,\quad \frac{\partial \Phi}{\partial n}=\D_{\partial\Omega}\ (An)\quad \textrm{on}\ \partial\Omega,   \tag{1.10} 
\end{equation*}
for skew-symmetric matrix-valued functions $A\in C^{\alpha}\big(\overline{\Omega}\big)$ for $\alpha\in (0,1)$, where $\D_{\partial\Omega}$ denotes the surface divergence on $\partial\Omega$ and $n=n_{\Omega}$ denotes the unit outward normal vector field on $\partial\Omega$. For a skew-symmetric $A$, $An$ is a tangential vector field on $\partial\Omega$. Moreover, $\D_{\partial\Omega}(An)= 0$ if $A$ is constant. We call $\Omega$ \textit{admissible} for $\alpha\in (0,1)$ if the a priori estimate 
\begin{equation*}
\sup_{x\in \Omega}d^{1-\alpha}_{\Omega}(x)
\big|\nabla \Phi(x)\big|\leq C\Big[A\Big]^{(\alpha)}_{\Omega}   \tag{1.11}
\end{equation*}
holds for all solutions of (1.10). Here, $d_{\Omega}(x)$ denotes the distance from $x\in \Omega$ to $\partial\Omega$. When $\alpha=0$, we replace the right-hand side by the sup-norm of $A$ on $\partial\Omega$ and call the corresponding notion admissible for $\alpha=0$. We call $\Omega$ \textit{strongly admissible} if $\Omega$ is admissible for all $\alpha\in [0,1)$. (See Definitions 2.1 and 2.3 in the next section). In this paper, we prove that bounded and exterior domains of class $C^{3}$ are strongly admissible.

The estimate (1.11) implies a scale invariant estimate corresponding to (1.9). We decompose $\nabla \Phi_0=\q_{\Omega}\partial f$ into two terms $\nabla \Phi_{1}=\q_{\R}\partial f$ and $\nabla \Phi_{2}$ (by the zero extension of $f$ to $\R\backslash\overline{\Omega}$). Then, $\Phi_{2}$ solves the Neumann problem (1.10) for $A=\nabla h-\nabla^{T}h$. We estimate $\nabla \Phi_{0}=\nabla \Phi_{1}+\nabla \Phi_{2}$ through the estimate (1.11) by
\begin{equation*}
\Big[\Phi_{1}\Big]^{(\alpha)}_{\R}+\sup_{x\in \Omega}d^{1-\alpha}_{\Omega}(x)
\big|\nabla \Phi_{2}(x)\big|\leq C_{\alpha}\Big[f \Big]^{(\alpha)}_{\Omega}.   \tag{1.12}
\end{equation*}
Since (1.12) is scale invariant, it is inherited to $\nabla \Phi_{0,m}=\q_{\Omega_m}\partial g_m$ so $\nabla \Phi_{0,m}$ tends to zero as $[g_m]^{(\alpha)}_{\Omega_{m}}\to 0$. This yields a sufficiently strong initial condition $u(\cdot,0)=0$ for the blow-up limit $(u,p)$. \\

Actually, we used the estimate (1.11) for $\alpha=0$ in order to prove analyticity of the Stokes semigroup on $C_{0,\sigma}$ by a similar blow-up argument \cite{AG1}. Since the pressure $p_m$ solves the Neumann problem (1.10) for $A=-\nabla u_m+\nabla^{T} u_m$, the a priori estimate (1.11) for $\alpha=0$ implies a scale invariant estimate for $\nabla p_m$ in terms of velocity on $L^{\infty}$ (harmonic-pressure gradient estimate). The harmonic-pressure gradient estimate implies a necessary time H\"older continuity of $p_m$ for the compactness of $(u_m,p_m)$ and a decay condition $\nabla p\to 0$ as $x_n\to\infty$ for the uniqueness of the blow-up limit $(u,p)$.\\

This paper is organized as follows. In Section 2, we define the term strongly admissible and prove that bounded and exterior domains of class $C^{3}$ are strongly admissible. In Section 3, we prove the H\"older-type estimate (1.12). In Section 4, we recall the $\tilde{L}^{p}$-theory and review the local H\"older estimates for the Stokes equations. In Section 5, we prove a uniqueness theorem for the Stokes equations in a half space. In Section 6, we prove Theorem 1.3. After the proof of Theorem 1.3, we complete the proof of Theorems 1.2 and 1.1. In Appendix A, we review $L^{1}$-type results for the Stokes equations in a half space used in Section 5. In Appendix B, we give short proofs for uniqueness of the heat equation used in Section 5 and Section 6.\\

\vspace{10pt} 
 
After the first draft of this paper is written \cite{AarXive1}, the author was informed of the paper \cite{Mar2014} on the exterior problem of the Navier-Stokes equations for $n\geq 3$. In the paper, a local solvability result for H\"older continuous initial data  \cite{GMZ} is extended for merely bounded $u_0\in L^{\infty}_{\sigma}$ based on the $L^{\infty}$-estimates of the Stokes semigroup \cite{AG1}, \cite{AG2}. Note that mild solutions of the Navier-Stokes equations on $L^{\infty}$ are not constructed without composition operator estimates. The estimate (1.5) yields unique existence of mild solutions on $C_{0,\sigma}$ for the exterior problem ($n\geq 2$) together with the existence time estimate from below by a sup-norm of initial data \cite{A3}. More recently, mild solutions on $L^{\infty}_{\sigma}$ are constructed in \cite{A4} based on the main results of this paper.

\vspace{10pt}

\section{Strongly admissible domains}

\vspace{5pt}

In this section, we introduce the term strongly admissible and prove that bounded and exterior domains of class $C^{3}$ are strongly admissible (Theorems 2.9 and 2.11). The proof is by a blow-up argument and parallel to the case $\alpha=0$ as in the previous works \cite{AG1}, \cite{AG2}.

\subsection{\rm{A priori estimates for the Neumann problem}}\mbox{}\\

Let $\Omega$ be a domain in $\R$, $n\geq 2$, $\partial\Omega\neq \emptyset$. We say that $\partial \Omega$ is $C^{k}$ ($k\geq 1$) if for each $x_{0}\in \partial \Omega$, there exist constants $\alpha,\beta, K$ and a $C^{k}$-function $h=h(y')$ such that by translation from $x_0$ to the origin and rotation, we have  
\begin{align*} 
&U(0) \cap \Omega=\Big\{(y',y_n)\in \R\ \big|\ h(y')<y_n<h(y')+\beta,\ |y'|<\alpha \Big\},   \\ 
&U(0) \cap \partial\Omega=\Big\{(y',y_n)\in \R\ \big|\ y_n=h(y'),\ |y'|<\alpha\Big\},\\
&\sup\limits_{|l|\leq k, |y'|<\alpha}\big|\partial^{l}_{y'} h(y')\big| \leq K,\  \nabla 'h(0)=0,\ h(0)=0,
\end{align*}
with the neighborhood of the origin 
\begin{equation*} 
U(0)=\Big\{(y',y_n) \in \R\ \big|\ h(y')-\beta<y_n<h(y')+\beta,\  |y'|<\alpha\Big\}.
\end{equation*}
Here, $\partial_x^l=\partial_{x_1}^{l_1} \cdots \partial_{x_n}^{l_n}$ for  a multi-index $l=(l_1, \ldots, l_n)$ and $\partial_{x_j}=\partial/\partial x_j$ as usual and $\nabla '$ denotes the gradient in $\mathbb{R}^{n-1}$. If $h$ is just Lipschitz continuous, we call $\partial\Omega$ Lipschitz boundary. Moreover, if we are able to take uniform constants $\alpha,\beta, K$ independent of each $x_0\in \partial\Omega$, we call $\Omega$ a  uniformly $C^{k}$-domain (Lipschitz domain) of type $\alpha,\beta,K$ as defined in \cite[I.3.2]{Sohr}. In order to distinguish $\alpha,\beta,K$ from H\"older exponents, we may write $\alpha', \beta', K'$.\\

We begin with the term admissible for $\alpha=0$. Let $\Omega$ be a domain in $\R$ with $C^{1}$-boundary. We consider the Neumann problem of the Laplace equation, 
\begin{equation*}
\Delta \Phi=0\quad \textrm{in}\ \Omega,\quad \frac{\partial \Phi}{\partial n}=\D_{\partial\Omega}\ (An)\quad \textrm{on}\ \partial\Omega,    \tag{2.1} 
\end{equation*}
for skew-symmetric matrix-valued functions $A$, where $\textrm{div}_{\partial\Omega}=\textrm{tr}\ \nabla_{\partial\Omega}$ denotes the surface divergence on $\partial\Omega$ and $\nabla_{\partial\Omega}=\nabla -n(n\cdot \nabla )$ is the gradient on $\partial\Omega$ for $n=n_{\Omega}$. Since $A=(a_{ij})$ is skew-symmetric, $An=(\sum_{j=1}^{n}a_{ij}n^{j})$ is a tangential vector field on $\partial\Omega$. Let $BC(\overline{\Omega})$ denote the space of all bounded and continuous functions in $\overline{\Omega}$. Let $BC_{\textrm{sk}}(\overline{\Omega})$ denote the space of all skew-symmetric matrix-valued functions $A\in BC(\overline{\Omega})$. We call $\Phi\in L^{1}_{\textrm{loc}}(\overline{\Omega})$ solution of (2.1) for $A\in BC_{\textrm{sk}}(\overline{\Omega})$ if $\Phi$ satisfies 
\begin{equation*}
\sup_{x\in \Omega}d_{\Omega}(x)\big|\nabla \Phi(x)\big|<\infty,   \tag{2.2}
\end{equation*}
and 
\begin{equation*}
\int_{\Omega}\Phi \Delta \varphi \textrm{d}x=\int_{\partial\Omega}An\cdot \nabla\varphi \dd{\cal{H}}    \tag{2.3}
\end{equation*}
for all $\varphi\in C^{2}_{c}\big(\overline{\Omega}\big)$ satisfying $\partial \varphi/\partial n=0$ on $\partial\Omega$, where $d{\cal{H}}$ denotes the surface element of $\partial\Omega$.  \\

The term admissible for $\alpha=0$ is defined by an a priori estimate for (2.1).

\vspace{5pt}

\begin{defn}[Admissible for $\alpha=0$]
Let $\Omega$ be a domain in $\R$ with $C^{1}$-boundary. We call $\Omega$ \textit{admissible for $\alpha=0$} if there exists a constant $C=C_{\Omega}$ such that the a priori estimate 
\begin{equation*}
\sup_{x\in\Omega}d_{\Omega}(x)\big|\nabla \Phi(x)\big|\leq C\big\|A\big\|_{L^{\infty}(\partial\Omega)}   \tag{2.4}
\end{equation*}
holds for all solutions of (2.1) for $A\in BC_{\textrm{sk}}(\overline{\Omega})$.    
\end{defn}

\begin{rem}
The term admissible was first introduced in \cite{AG1} by using  the Helmholtz projection $\p$ on $\tilde{L}^{p}=L^{p}\cap L^{2}$ for uniformly $C^{1}$-domains. We call $\Omega$ admissible  in the sense of \cite[Definition 2.3]{AG1} if there exists a constant $C$ such that the estimate 
\begin{equation*}
\sup_{x\in \Omega}d_{\Omega}(x)\big|(\q_{\Omega}\nabla \cdot f)(x)\big|\leq C\big\|f\big\|_{L^{\infty}(\partial\Omega)}   \tag{2.5}
\end{equation*}
holds for all matrix-valued functions $f=(f_{ij})\in C^{1}(\overline{\Omega})$ such that $\nabla \cdot f=(\sum_{j}\partial_{j}f_{ij})\in \tilde{L}^{p}$ ($p\geq n$), $\textrm{tr}\ f=0$ and $\partial_{l}f_{ij}=\partial_{j}f_{il}$ for $i,j,l\in \{1,\cdots, n \}$. The term admissible for $\alpha=0$ is a stronger notion than admissible in the sense of \cite{AG1}. In fact, $\nabla \Phi=\q_{\Omega}\nabla \cdot f$ is a solution of the Neumann problem (2.1) for $A=-f+f^{T}$ so the estimate (2.5) follows from (2.4). Although admissible for $\alpha=0$ is stronger than the original notion, we are able to prove that bounded and exterior domains of class $C^{3}$ are also admissible for $\alpha=0$ by a blow-up argument as in \cite{AG1}, \cite{AG2} (see also Remark 2.10). 
\end{rem}

\vspace{5pt}

We define the term admissible for $\alpha\in (0,1)$. Let  $C^{\alpha}\big(\overline{\Omega}\big)$ denote the space of all H\"older continuous functions with exponent $\alpha$ in $\overline{\Omega}$. Let $C_{\textrm{sk}}^{\alpha}\big(\overline{\Omega}\big)$ denote the space of all skew-symmetric matrix-valued functions $A\in C^{\alpha}\big(\overline{\Omega}\big)$. We call $\nabla \Phi\in L^{1}_{\textrm{loc}}(\overline{\Omega})$ solution of (2.1) for $A\in C^{\alpha}_{\textrm{sk}}\big(\overline{\Omega}\big)$ if $\Phi$ satisfies 
\begin{equation*}
\sup_{x\in \Omega}d^{1-\alpha}_{\Omega}(x)\big|\nabla \Phi(x)\big|<\infty,    \tag{2.6}
\end{equation*}
and 
\begin{equation*}
\int_{\Omega}\nabla \Phi \cdot\nabla \varphi \textrm{d}x=-\int_{\partial\Omega}An\cdot \nabla\varphi \dd{\cal{H}}    \tag{2.7}
\end{equation*}
for all $\varphi\in C^{1}_{c}\big(\overline{\Omega}\big)$. We also call $\nabla \Phi$ for $A\in C^{\alpha}_{\textrm{sk}}\big(\overline{\Omega}\big)$ solution of (2.1) of type $\alpha$ in order to distinguish it from that for $A\in BC_{\textrm{sk}}(\overline{\Omega})$.\\

We define the term strongly admissible by a priori estimates for $\alpha\in [0,1)$. 

\vspace{5pt}

\begin{defn}[Strongly admissible]
Let $\Omega$ be a domain in $\R$ with $C^{1}$-boundary. We call $\Omega$ admissible for $\alpha \in (0,1)$ if there exists a constant $C=C_{\alpha,\Omega}$ such that the a priori estimate 
\begin{equation*}
\sup_{x\in \Omega}d^{1-\alpha}_{\Omega}(x)\big|\nabla \Phi(x)\big|\leq C\Big[A\Big]^{(\alpha)}_{\Omega}    \tag{2.8}
\end{equation*}
holds for all solutions of (2.1) for $A\in {C}^{\alpha}_{\textrm{sk}}\big(\overline{\Omega}\big)$. We call $\Omega$ \textit{strongly admissible} if $\Omega$ is admissible for all $\alpha\in [0,1)$. 
\end{defn}

\begin{rems}
\noindent
(i) The constants in (2.4) and (2.8) are invariant of dilation,  translation and rotation of $\Omega$.

\noindent 
(ii) Strongly admissible domains are admissible in the sense of \cite{AG1} by Remark 2.2.

\noindent 
(iii) A half space is strongly admissible. Let $E$ denote the fundamental solution of the Laplace equation, i.e., $E(x)=C_{n}/ |x|^{n-2}$ for $n\geq 3$ and $E(x)=-1/(2\pi)\log{|x|}$ for $n=2$, where $C_n=(an(n-2))^{-1}$ and the volume of $n$-dimensional unit ball $a$. Solutions of the Neumann problem (2.1) are expressed by 
\begin{equation*}
\Phi(x',x_n)=\int_{x_n}^{\infty}e^{sA}\D_{\partial\R_{+}}w \dd s
\end{equation*}
for $w=An_{\R_{+}}$ and the Poisson semigroup 
\begin{equation*}
e^{sA}g=-2\int_{\partial\R_{+}}\frac{\partial E}{\partial s}(x'-y',s)g(y')\dd y'.
\end{equation*}
Here, $x'$ denotes $(n-1)$-variable of $x=(x',x_n)$. The Poisson semigroup is an analytic semigroup on $L^{p}(\mathbb{R}^{n-1})$ for $1\leq p\leq \infty$ and its generator is given by $A=-(-\Delta_{\textrm{tan}})^{1/2}$ (see, e.g., \cite[Example 3.7.9]{ABHN}). The a priori estimates (2.4) and (2.8) can be viewed as the $L^{\infty}$-estimates of the Poisson semigroup  
\begin{align*}
\big\|\partial_{\textrm{tan}}e^{sA}g\big\|_{L^{\infty}(\mathbb{R}^{n-1})}
&\leq \frac{C}{s}\big\|g\big\|_{L^{\infty}(\mathbb{R}^{n-1})}, \tag{2.9} \\
\big\|\partial_{\textrm{tan}}e^{sA}g\big\|_{L^{\infty}(\mathbb{R}^{n-1})}
&\leq \frac{C}{s^{1-\alpha}}\Big[g\Big]^{(\alpha)}_{\mathbb{R}^{n-1}}
\quad s>0.  \tag{2.10}
\end{align*}  
Here, $\partial_{\textrm{tan}}=\partial_{j}$ indiscriminately denotes the tangential derivatives $j=1,\cdots,n-1$. The estimates (2.4) and (2.8) respectively follow from (2.9) and (2.10). Since 
\begin{align*}
\partial_{x_{j}} e^{sA}g
&= -2\int_{\partial\R_{+}}\frac{\partial^{2} E}{\partial x_{j}\partial s}(x'-y',s)g(y')\dd y'  \\
&=-2\int_{\partial\R_{+}}\frac{\partial^{2} E}{\partial y_{j}\partial s}(y',s)g(x'-y')\dd y',\\
&=-2\int_{\partial\R_{+}}\frac{\partial^{2} E}{\partial y_{j}\partial s}(y',s)\big(g(x'-y')-g(x')  \big)\dd y',
\end{align*}
it follows that
\begin{align*}
\big\|\partial_{\textrm{tan}} e^{sA}g\big\|_{L^{\infty}(\mathbb{R}^{n-1})}
&\leq C\Big[g\Big]^{(\alpha)}_{\mathbb{R}^{n-1}}\int_{\partial\R_{+}}\frac{|y'|^{\alpha}}{(|y'|^{2}+s^{2})^{\frac{n}{2}}}\dd y'\\
&= \frac{C'}{s^{1-\alpha}}\Big[g\Big]^{(\alpha)}_{\mathbb{R}^{n-1}}.
\end{align*}
Thus, (2.9) and (2.10) hold.

\noindent 
(iv) For a skew-symmetric constant matrix $A=(a_{ij})$, the surface divergence of $An$ vanishes, i.e., $\D_{\partial\Omega}\ (An)=0$, in the sense that  
\begin{equation*}
\int_{\partial\Omega}An\cdot \nabla\varphi\dd{\cal{H}}=0  \quad \textrm{for}\ \varphi\in C^{1}_{c}\big(\overline{\Omega}\big).   
\end{equation*}
In fact, it follows that 
\begin{align*}
\int_{\partial\Omega}An\cdot \nabla \varphi \dd{\cal{H}}
&=\sum_{i,j}\int_{\partial\Omega}a_{ij}n^{j}\partial_{i}\varphi\dd {\cal{H}}   \\
&=\sum_{i,j}\int_{\Omega}a_{ij}\partial_{j}\partial_{i}\varphi\dd x  \\
&=\sum_{i,j}\int_{\Omega}a_{ji}\partial_{i}\partial_{j}\varphi\dd x
=-\sum_{i,j}\int_{\Omega}a_{ij}\partial_{j}\partial_{i}\varphi\dd x.
\end{align*}
The right-hand side equals zero.
\end{rems}

\vspace{5pt}

\subsection{\rm{Uniqueness of the Neumann problem}}\mbox{}\\

We prove the uniqueness of the Neumann problem (2.1) in order to prove the a priori estimate (2.8) by a blow-up argument. 

\vspace{5pt}

\begin{lem}
Let $\nabla \Phi\in L^{1}_{\textrm{loc}}(\overline{\R_{+}})$ satisfy 
\begin{equation*}
\int_{\R_{+}}\nabla \Phi\cdot \nabla \varphi \dd x=0\quad \textrm{for all}\ \varphi\in C^{1}_{c}(\overline{\R_{+}}).    \tag{2.11}
\end{equation*}
Assume that 
\begin{equation*}
\sup_{x\in \R_{+}}x_{n}^{1-\alpha}\big|\nabla \Phi(x)\big|<\infty,    \tag{2.12}
\end{equation*}
for some $\alpha \in (0,1)$. Then, $\nabla \Phi\equiv 0$.
\end{lem}

\begin{proof}
We consider the even extension of $\Phi$ to $\R$, i.e.,
\begin{equation*}
\tilde{\Phi}(x',x_n)=
\begin{cases}
&\Phi(x',x_n)\quad \hspace{7pt} \textrm{for}\  x_n\geq 0,  \\
&\Phi(x',-x_n)\quad \textrm{for}\  x_n<0.
\end{cases}
\end{equation*}
Then, $\tilde{\Phi}$ is weakly harmonic in $\R$. In fact, for $\tilde{\varphi}\in C^{2}_{c}(\R)$, the function $\tilde{\Phi}$ satisfies 
\begin{align*}
\int_{\R}\tilde{\Phi}(x)\Delta \tilde{\varphi}(x)\dd x
&=\int_{0}^{\infty}\int_{\mathbb{R}^{n-1}}{\Phi}(x',x_n)\Delta \tilde{\varphi}(x',x_n)\dd x
+\int_{-\infty}^{0}\int_{\mathbb{R}^{n-1}}{\Phi}(x',-x_n)\Delta \tilde{\varphi}(x',x_n)\dd x  \\
&=\int_{0}^{\infty}\int_{\mathbb{R}^{n-1}}{\Phi}(x',x_n)\Delta{\varphi}(x',x_n)\dd x.    
\end{align*}
Since $\varphi(x',x_n)=\tilde{\varphi}(x',x_n)+\tilde{\varphi}(x',-x_n)$ is $C^{2}$ in $\overline{\R_{+}}$ and satisfies $\partial \varphi/\partial x_n=0$ on $\{x_n=0\}$, the right-hand side equals zero by (2.11). Thus, $\tilde{\Phi}\in L^{1}_{\textrm{loc}}(\R)$ is weakly harmonic in $\R$. By Weyl's lemma, $\tilde{\Phi}$ is smooth in $\mathbb{R}^{n}$. By (2.12), $\nabla \tilde{\Phi}$ is bounded in $\R$ and decays as $x_n\to\infty$. We apply the Liouville theorem and conclude that $\nabla \tilde{\Phi}\equiv 0$. 
\end{proof}

\vspace{5pt}

We next prove the uniqueness theorem for bounded and exterior domains. Note that $\nabla \Phi\in L^{p}_{\textrm{loc}}\big(\overline{\Omega}\big)$ for $1\leq p<1/(1-\alpha)$ provided that  $d_{\Omega}^{1-\alpha}\nabla \Phi\in L^{\infty}(\Omega)$. In particular, $\nabla\Phi\in L^{p}(\Omega)$ when $\Omega$ is bounded.

\vspace{5pt}

\begin{lem}
Let $\Omega$ be a bounded domain in $\R$ with $C^{2}$-boundary. Let $\nabla \Phi\in L^{1}_{\textrm{loc}}(\overline{\Omega})$ satisfy 
\begin{equation*}
\int_{\Omega}\nabla \Phi\cdot \nabla \varphi \dd x=0\quad   \textrm{for all}\ \varphi\in C^{1}\big(\overline{\Omega}\big). \tag{2.13}
\end{equation*}
Assume that 
\begin{equation*}
\sup_{x\in \Omega}d_{\Omega}^{1-\alpha}(x)\big|\nabla \Phi(x)\big|<\infty,     \tag{2.14}
\end{equation*}
for some $\alpha \in (0,1)$. Then, $\nabla \Phi\equiv 0$.
\end{lem}

\begin{proof}
We consider the Neumann problem,
\begin{align*}
\Delta \varphi=\D\ g \quad \textrm{in}\ \Omega,  \\
\frac{\partial \varphi}{\partial n}=0\quad \textrm{on}\ \partial\Omega.
\end{align*}
For $g\in C^{\infty}_{c}(\Omega)$, there exists a solution $\varphi\in W^{2,q}(\Omega)$ for $q\in (1,\infty)$ (e.g., \cite[Teor. 4.1]{LMp}). In particular, $\varphi$ is in $C^{1}\big(\overline{\Omega}\big)$. Since $\nabla \Phi\in L^{p}(\Omega)$ for $1\leq p<1/(1-\alpha)$ by (2.14), it follows that
\begin{equation*}
\int_{\Omega}\Phi\D\ g\dd x
=\int_{\Omega}\Phi\Delta \varphi\dd x 
=-\int_{\Omega}\nabla \Phi\cdot \nabla \varphi\dd x=0.
\end{equation*} 
We proved $\nabla \Phi\equiv 0$. 
\end{proof}

\vspace{5pt}

\begin{lem}
Let $\Omega$ be an exterior domain in $\R$, $n\geq 2$, with $C^{2}$-boundary. Let $\nabla \Phi\in L^{1}_{\textrm{loc}}(\overline{\Omega})$ satisfy 
\begin{equation*}
\int_{\Omega}\nabla \Phi\cdot \nabla \varphi \dd x=0\quad\textrm{for all}\ \varphi\in C^{1}_{c}\big(\overline{\Omega}\big).  \tag{2.15}
\end{equation*}
Assume that 
\begin{equation*}
\sup_{x\in \Omega}d_{\Omega}^{1-\alpha}(x)\big|\nabla \Phi(x)\big|<\infty,    \tag{2.16}
\end{equation*}
for some $\alpha \in (0,1)$. Then, $\nabla \Phi\equiv 0$.
\end{lem}

\begin{proof}
We first estimate $\Phi(x)$ as $|x|\to \infty$ by using (2.16). We may assume $0\in \Omega^{c}$. We take a constant $R_{0}>\textrm{diam}\ \Omega^{c}$ and observe that $|x|\leq 2d_{\Omega}(x)$ for $|x|\geq 2R_0$. It follows from (2.16) that   
\begin{equation*}
\sup_{|x|\geq 2R_0}|x|^{1-\alpha}\big|\nabla \Phi(x)\big|<\infty.    
\end{equation*}    
By a fundamental calculation, we estimate 
\begin{equation*}
\big|\Phi(x)\big|\leq C_1|x|^{\alpha}+C_2\quad \textrm{for}\ |x|\geq 2R_0,      \tag{2.17}
\end{equation*}
with some constants $C_1$ and $C_2$ independent of $x$.

We consider the Neumann problem,
\begin{align*}
\Delta \varphi=\D\ g \quad \textrm{in}\ \Omega,  \\
\frac{\partial \varphi}{\partial n}=0\quad \textrm{on}\ \partial\Omega.
\end{align*}
For $g\in C^{\infty}_{c}(\Omega)$, there exists a solution $\varphi\in L^{q}_{\textrm{loc}}(\overline{\Omega})$ satisfying $\nabla \varphi\in L^{q}$ for $q\in (1,\infty)$ \cite{Gal}. By the elliptic regularity theory \cite{LMp}, we have $\varphi\in W^{2,q}_{\textrm{loc}}(\overline{\Omega})$. In particular, $\varphi$ is $C^{1}$ in $\overline{\Omega}$. In order to substitute $\varphi$ into (2.15), we cutoff the function $\varphi$ as $|x|\to\infty$. Let $\theta\in C^{\infty}_{c}[0,\infty)$ be a smooth cut-off function satisfying $\theta\equiv 1$ in $[0,1]$ and $\theta\equiv 0$ in $[2,\infty)$. Set $\theta_{m}(x)=\theta(|x|/m)$ for $m\geq R_{0}$ so that $\theta_{m}\equiv 1$ for $|x|\leq m$, $\theta_{m}\equiv 0$ for $|x|\geq 2m$ and $\textrm{spt}\ \nabla\theta_{m}\subset \overline{D_{m}}$ for $D_{m}=\{m<|x|<2m\}$. Since $\varphi_{m}=\varphi\theta_{m}\in C^{1}_{c}(\overline{\Omega})\cap W^{2,q}_{\textrm{loc}}(\overline{\Omega})$ satisfies $\partial\varphi_{m}/\partial n=0$ on $\partial\Omega$ and
\begin{align*}
\Delta \varphi_{m}
&=\Delta \varphi\theta_{m}+2\nabla\varphi\cdot\nabla \theta_{m}+\varphi\Delta\theta_{m}  \\
&=\D\ g_m-g\cdot\nabla \theta_{m}+2\nabla\varphi\cdot\nabla \theta_{m}+\varphi\Delta\theta_{m}
\end{align*}
for $g_m=g\theta_{m}$, it follows from (2.15) that
\begin{align*}
\int_{\Omega}\Phi\D\ g_m\dd x
&=\int_{\Omega}\Phi\left(\Delta \varphi_{m}+g\cdot\nabla \theta_{m}-2\nabla\varphi\cdot\nabla \theta_{m}-\varphi\Delta\theta_{m}\right)\dd x\\
&=\int_{\Omega}\Phi\left(g\cdot\nabla \theta_{m}-2\nabla\varphi\cdot\nabla \theta_{m}-\varphi\Delta\theta_{m}\right)\dd x
=:I_m+II_m+III_m.
\end{align*}
Since $g$ is compactly supported in  $\Omega$, $g_m\equiv g$ and $I_m\equiv 0$ for sufficiently large $m\geq R_0$. We shall show that $II_m,III_m\to 0$ as $m\to\infty$. By the cut-off function estimate $||\nabla^{k}\theta_{m}||_{\infty}\leq C/m^{|k|}$ for $|k|\geq 0$ and (2.17), it follows that 
\begin{equation*}
\big|II_m\big|\leq \frac{C}{m^{1-\alpha-\frac{n}{q'}}}\big\|\nabla \varphi\big\|_{L^{q}(D_m)},
\end{equation*}
with the constant $C$, independent of $m\geq 2R_0$, where $1/q+1/q'=1$.

By a similar way, we estimate $III_{m}$. By the Poincar\'e inequality \cite[5.8.1]{E}, we estimate 
\begin{equation*}
\big\|\varphi-(\varphi)\big\|_{L^{q}(D_m)}\leq mC\big\|\nabla \varphi\big\|_{L^{q}(D_m)}
\end{equation*}   
with some constant $C$ independent of $m$, where $(\varphi)$ denotes the average of $\varphi$ in $D_m$. Since $\Delta \theta_{m}$ is supported in $\overline{D_m}$, it follows that 
\begin{equation*}
\big|III_{m}\big|=\left|\int_{\Omega}\Phi\big(\varphi-(\varphi)\big)\Delta \theta_{m}\dd x\right|\leq \frac{C}{m^{1-\alpha-\frac{n}{q'}}}\big\|\nabla \varphi\big\|_{L^{q}(D_m)}.
\end{equation*}
The function $\nabla\varphi$ is $L^{q}$-integrable in $\Omega$ for all $q\in (1,\infty)$. In particular, $\nabla\varphi\in L^{q}$ for $q\in (1,n/(n-1+\alpha)]$ and $1-\alpha-n/q'\geq 0$. Thus, $|II_m|+|III_m|\to0$ as $m\to\infty$. We proved $\nabla\Phi\equiv 0$. The proof is now complete.
\end{proof}

\vspace{5pt}

In the next subsection, we apply the following extension theorem of harmonic functions in order to prove the a priori estimate (2.8) for exterior domains by a blow-up argument.

\vspace{5pt}

\begin{prop}
Let $\Phi$ be a harmonic function in $\R\backslash \{0\}$, $n\geq 2$. 
Let $\alpha\in (0,1)$. Assume that 
\begin{equation*}
\sup\left\{|x|^{1-\alpha}\big|\nabla \Phi(x)\big|\ \Big|\ x\in B_{0}(1),\ x\neq 0  \right\}<\infty.
\end{equation*}
Then, $\Phi$ is extendable to a harmonic function in $\R$.
\end{prop}

\begin{proof}
The assertion is proved for $n\geq 3$ under the weaker assumption $\alpha=0$ in \cite[Lemma A.1]{AG2}. When $n=2$, $\Phi=\log{|x|}$ satisfies $\Phi=O(|x|^{-1})$ as $|x|\to 0$ and the statement for $\alpha=0$ fails. It is proved in \cite{AG2} by a cut-off function argument that a harmonic function $\Phi$ in $\mathbb{R}^{2}\backslash \{0\}$ is extendable to a harmonic function in $\mathbb{R}^{2}$ if $\nabla \Phi=O(|x|^{-1})$ and the spherical mean of $\Phi$ over the sphere is independent of $r>0$, i.e., $\fint_{\partial B_{0}(r)}\Phi\dd \mathcal{H}=\textrm{constant}$ for $r<1$. Under the stronger assumption $\nabla \Phi=O(|x|^{-\alpha})$, the spherical mean condition is removable and the cut-off function argument applies to prove that $\Phi$ is extendable to a harmonic function in $\mathbb{R}^{2}$ without modification. 
\end{proof}

\vspace{5pt}

\subsection{\rm{Blow-up arguments}}\mbox{}\\

Since bounded and exterior domains of class $C^{3}$ are admissible for $\alpha=0$ as in Remark 2.2, we prove the a priori estimate (2.8) for $\alpha\in (0,1)$. 

\vspace{5pt}

\begin{thm}
A bounded domain of class $C^{3}$ is strongly admissible.
\end{thm}

\begin{proof}
We argue by contradiction. Suppose that (2.8) were false for any choice of constants $C$. Then, there would exist a sequence of solutions of (2.1), $\tilde{\Phi}_{m}$ for $\tilde{A}_{m}\in C^{\alpha}_{\textrm{sk}}(\overline{\Omega})$ such that 
\begin{equation*}
M_{m}=\sup_{x\in \Omega}d_{\Omega}^{1-\alpha}(x)\big|\nabla\tilde{\Phi}_{m}(x)\big|>m\Big[\tilde{A}_{m}\Big]^{(\alpha)}_{\Omega}.
\end{equation*}
Divide both sides by $M_m$ and observe that $\Phi_{m}=\tilde{\Phi}_{m}/M_m$ and $A_m=\tilde{A}_{m}/M_m$ satisfy
\begin{align*}
&\sup_{x\in \Omega}d_{\Omega}^{1-\alpha}(x)\big|\nabla\Phi_{m}(x)\big|=1,     \tag{2.18}\\
&\big[{A}_{m}\big]^{(\alpha)}_{\Omega}< \frac{1}{m}.   \tag{2.19}
\end{align*}
We take a point $x_{m}\in \Omega$ such that 
\begin{equation*}
d_{\Omega}^{1-\alpha}(x_m)\big|\nabla \Phi_{m}(x_m)\big|\geq \frac{1}{2}.   \tag{2.20}
\end{equation*}
Since $\Omega$ is bounded, there exists a subsequence of $\{x_m\}\subset \Omega$ (still denoted by $\{x_m\}$) such that $x_m\to x_{\infty}\in \overline{\Omega}$ as $m\to\infty$. Then, the proof is divided into two cases whether $x_{\infty}\in \Omega$ or $x_{\infty}\in \partial\Omega$.\\

\noindent 
\textit{Case 1}\ $x_{\infty}\in \Omega$. We take a point $x_0\in \Omega$ and observe from (2.19) that $\hat{A}_{m}(x)=A_{m}(x)-A_{m}(x_0)$ converges to zero uniformly in $\overline{\Omega}$ as $m\to\infty$. Since $A_{m}(x_0)$ is skew-symmetric, we replace $A_{m}$ to $\hat{A}_{m}$, i.e., 
\begin{equation*}
\int_{\Omega}\nabla \Phi_{m}\cdot \nabla\varphi\dd x
=-\int_{\partial\Omega}\hat{A}_{m}n\cdot \nabla\varphi \dd {\cal{H}}
\end{equation*}
for all $\varphi\in C^{1}(\overline{\Omega})$ by Remarks 2.4 (iv). Since $\Phi_{m}$ is harmonic in $\Omega$, $\nabla \Phi_{m}$ subsequently converges to $\nabla\Phi$ locally uniformly in $\Omega$. In fact, by the mean-value formula $\nabla \Phi_m(x)=\fint_{B_{x}(r)}\nabla \Phi_m(y)\dd y$ and integration by parts, we have 
\begin{align*}
\big|\nabla^{2} \Phi_{m}(x)\Big|\leq \frac{n}{r}\big\|\nabla \Phi_m\big\|_{L^{\infty}(\partial B_{x}(r))}
\end{align*}
for $x\in \Omega$ and $r>0$ such that $B_{x}(r)\subset \Omega$. Thus, $\nabla^{2}\Phi_m$ is uniformly bounded in $B_{x}(r)$ by (2.18). We consider an arbitrary compact set $K\subset \Omega$ and, by taking a finite number of open balls, obtain a uniform bound of $\nabla^{2}\Phi_m$ in $K$. Since $\nabla \Phi_m$ is equi-continuous in $K$, by choosing a subsequence, $\nabla \Phi_m$ converges to $\nabla \Phi$ uniformly in $K$ by Ascoli-Arzel\`a theorem. We observe a convergence of the integral. Since $\nabla \Phi_m$ is bounded in $L^{p}(\Omega)$ for $1\leq p<1/(1-\alpha)$ by (2.18), by choosing a subsequence, we have $\nabla \Phi_m\to \nabla \Phi$ weakly in $L^{p}(\Omega)$. Sending $m\to\infty$ implies
\begin{equation*}
\int_{\Omega}\nabla \Phi\cdot \nabla\varphi\dd x=0.
\end{equation*}
We apply Lemma 2.6 and conclude that $\nabla \Phi\equiv 0$. This contradicts $d_{\Omega}^{1-\alpha}(x_{\infty})|\nabla \Phi(x_{\infty})|\geq 1/2$ by (2.20). So Case 1 does not occur.\\

\noindent 
\textit{Case 2}\ $x_{\infty}\in \partial\Omega$. Since the points $\{x_m\}$ accumulate the boundary and $\partial\Omega$ is sufficiently regular, there exists a unique projection $\tilde{x}_{m}\in \partial\Omega$ such that $x_m=\tilde{x}_{m}-d_m n_{\Omega}(\tilde{x}_{m})$. By translation and rotation around $\tilde{x}_m\in \partial\Omega$, we may assume that $x_{m}=(0,\cdots,0,d_m)$ and $\tilde{x}_{m}=0$. We rescale $\Phi_{m}$ around the point $x_m\in \Omega$ to get a blow-up sequence,
\begin{align*}
\Psi_{m}(x)&=d_m^{-\alpha}\Phi_{m}(x_{m}+d_{m}x),   \\
B_{m}(x)&=d_m^{-\alpha}A_{m}(x_{m}+d_{m}x).
\end{align*} 
The blow-up sequence $\Psi_{m}$ satisfies the Neumann problem  (2.1) for $B_{m}\in C^{\alpha}_{\textrm{sk}}\big(\overline{\Omega_m}\big)$ in the rescaled domain 
\begin{equation*}
\Omega_{m}=\frac{\Omega-\{x_m\}}{d_m}.
\end{equation*}
Note that the distance from the origin to the boundary $\partial\Omega_{m}$ is normalized as one, i.e., $d_{\Omega_{m}}(0)=1$ since we rescale $\Phi_{m}$ by $d_m=d_{\Omega}(x_m)$. The rescaled domain $\Omega_m$ expands to the half space $\R_{+,-1}=\{(x',x_n)\in \R\ |\ x_n>-1\}$. In fact, we consider the neighborhood of $x_{\infty}=0$,
\begin{equation*}
\Omega_{\textrm{loc}}=\Big\{(x',x_n)\in \R\ \Big|\ h(x')<x_n<h(x')+\beta',\ |x'|<\alpha'\Big\},
\end{equation*}
with some constants $\alpha'$, $\beta'$, $K'$ and the $C^{2}$-function $h$ satisfying $h(0)=0$, $\nabla'h(0)=0$ and $||h||_{C^{2}(\{  |x'|<\alpha'\})}\leq K'$. Then, $\Omega_{\textrm{loc} }\subset \Omega$ is rescaled to 
\begin{equation*}
\Omega_{\textrm{loc},m}=\left\{(x',x_n)\in \R\ \Bigg|\ h_{m}(x')-1<x_n<h_{m}(x')-1+\frac{\beta'}{d_m},\ |x'|<\frac{\alpha'}{d_m}     \right\},
\end{equation*}
where $h_{m}(x')=d_{m}^{-1}h(d_mx')$. Since $\nabla' h(0)=0$, $\Omega_{\textrm{loc},m}$ expands to the half space $\R_{+,-1}$ as $m\to\infty$. 

We take an arbitrary $\varphi\in C^{1}_{c}(\overline{\R_{+,-1}})$ and extend it as a compactly supported $C^{1}$-function in $\R$ (see, e.g., \cite[5.4]{E}). Then, $\Psi_{m}$ and $B_m$ satisfy 
\begin{equation*}
\int_{\Omega_{m}}\nabla \Psi_{m}\cdot \nabla \varphi\dd x
=-\int_{\partial\Omega_{m}}{B}_{m}n_{\Omega_{m}}\cdot \nabla \varphi \dd {\cal{H}}.    \tag{2.21}
\end{equation*}
The estimates (2.18)--(2.20) are inherited to 
\begin{equation*}
\begin{aligned}
&\sup_{x\in \Omega_m}d_{\Omega_{m}}^{1-\alpha}(x)\big|\nabla \Psi_{m}(x)\big|\leq 1,  \\
&\big[B_{m} \big]^{(\alpha)}_{\Omega_{m}}< \frac{1}{m},  \\
&\big|\nabla\Psi_{m}(0)\big|\geq \frac{1}{2}.
\end{aligned}
\end{equation*}
We set $\hat{B}_{m}(x)=B_{m}(x)-B_{m}(x_0)$ by some $x_{0}\in \Omega_m$. Then, $\hat{B}_m$ satisfies
\begin{equation*}
\big|\hat{B}_{m}(x)\big|\leq \frac{1}{m}|x-x_0|^{\alpha}\quad \textrm{for}\ x\in \Omega_{m}.
\end{equation*}
Since $B_m(x_0)$ is skew-symmetric, we replace $B_{m}$ to $\hat{B}_{m}$ in (2.21). Since $\nabla \Psi_m$ is harmonic and $d_{\Omega_m}^{1-\alpha}\nabla \Psi_m$ is uniformly bounded in $\Omega_m$, by choosing a subsequence, $\nabla \Psi_m$ converges to $\nabla \Psi$ locally uniformly in $\mathbb{R}^{n}_{+,-1}$ as in Case 1. We observe a convergence of the integral (2.21). Since local $L^{p}$-norms of $\nabla \Psi_m$ in $\overline{\Omega_m}$ are uniformly bounded for $1\leq p<1/(1-\alpha)$, by choosing a subsequence, we have $\nabla \Psi_m\to \nabla \Psi$ weakly in $L^{p}_{\textrm{loc}}(\overline{\R_{+,-1}})$. Sending $m\to\infty$ implies
\begin{equation*}
\int_{\R_{+,-1}}\nabla \Psi\cdot \nabla\varphi\dd x=0.
\end{equation*}
We apply Lemma 2.5 and conclude that $\nabla \Psi\equiv 0$. This contradicts $|\nabla\Psi(0)|\geq 1/2$. Thus, Case 2 does not occur.

We reached a contradiction. The proof is now complete.
\end{proof}

\begin{rem}(Boundary regularity)
We are able to prove the a priori estimate (2.4) for $C^{3}$-bounded domains by a similar blow-up argument (see \cite{AG1}). Since $\nabla \Phi$ may not be integrable up to the boundary under the bound (2.2), we used the weak form (2.3). This is the reason why we need $C^{3}$ to prove (2.4) by a blow-up argument. However, $C^{3}$ is not optimal for (2.4). In fact, it is proved in \cite[Lemma 6.2]{KLS} by using the Green function that (2.4) holds for $C^{1,\gamma}$-bounded domains. Thus, bounded domains of class $C^{2}$ are also strongly admissible.  
\end{rem}

\vspace{5pt}

\begin{thm}
An exterior domain of class $C^{3}$ is strongly admissible.
\end{thm}

\begin{proof}
We argue by contradiction. Suppose that (2.8) were false. Then, there would exist a sequence of solutions for (2.1), $\Phi_{m}$ for $A_m\in C^{\alpha}_{\textrm{sk}}(\overline{\Omega})$ and a sequence of points $\{x_m\}\subset \Omega$ such that 
\begin{align*}
&\sup_{x\in \Omega}d_{\Omega}^{1-\alpha}(x)\big|\nabla \Phi_{m}(x)\big|\leq 1,  \tag{2.22}\\
&\big[A_{m} \big]^{(\alpha)}_{\Omega}< \frac{1}{m}, \tag{2.23} \\
&d_{\Omega}^{1-\alpha}(x_m)\big|\nabla\Phi_{m}(x_m)\big|\geq \frac{1}{2}.   \tag{2.24}
\end{align*}
The proof is divided into two cases depending on whether $d_m=d_{\Omega}(x_m)$ converges or not.\\

\noindent 
\textit{Case 1} $\overline{\lim}_{m\to\infty}d_{m}<\infty$. We may assume $x_m\to x_{\infty}\in \overline{\Omega}$ as $m\to\infty$ by choosing a subsequence. Then, Case 1 is divided into two cases whether $x_{\infty}\in \Omega$ or $x_{\infty}\in \partial\Omega$.\\

\noindent
(i) $x_{\infty}\in \Omega$. The proof reduces to the uniqueness in the exterior domain $\Omega$. As in Case 1 in the proof of Theorem 2.9, there exists a subsequence of $\{\Phi_{m}\}$ (still denoted by $\{\Phi_{m}\}$) such that $\nabla \Phi_m\to \nabla \Phi$ locally uniformly in $\Omega$. Moreover, by choosing a subsequence, $\nabla \Phi_m\to \nabla \Phi$ weakly in $L^{p}_{\textrm{loc}}(\overline{\Omega})$ for $1\leq p <1/(1-\alpha)$. By (2.23), sending $m\to\infty$ implies
\begin{equation*}
\int_{\Omega}\nabla\Phi\cdot \nabla\varphi\dd x=0\quad\textrm{for all}\ \varphi\in C^{1}_{c}(\overline{\Omega}).
\end{equation*}
We apply Lemma 2.7 and conclude that $\nabla\Phi\equiv 0$. This contradicts $d_{\Omega}^{1-\alpha}(x_{\infty})|\nabla \Phi(x_{\infty})|\geq 1/2$. So the case (i) does not occur.\\

\noindent 
(ii) $x_{\infty}\in \partial\Omega$. The proof reduces to the uniqueness in a half space. By the same rescaling argument as in Case 2 of the proof of Theorem 2.9, we are able to prove that the case (ii) does not occur.\\

\noindent 
\textit{Case 2} $\overline{\lim}_{m\to\infty}d_{m}=\infty$. We may assume ${\lim}_{m\to\infty}d_{m}=\infty$. The proof reduces to the whole space. We rescale $\Phi_{m}$ around the point $x_m\in \Omega$ to get a blow-up sequence,
\begin{equation*}
\Psi_{m}(x)=d_m^{-\alpha}\Phi_{m}(x_{m}+d_{m}x)\quad \textrm{for}\ x\in \Omega_{m}:=\frac{\Omega-\{x_m\}}{d_m}.   
\end{equation*} 
Since we rescale $\Phi_{m}$ by $d_m=d_{\Omega}(x_m)$, the distance from the origin to $\partial\Omega_m$ is normalized as one, i.e., $d_{\Omega_m}(0)=1$. We take a point $a_m\in \partial\Omega_{m}$ such that $|a_m|=d_{\Omega_m}(0)=1$. By choosing a subsequence of $\{a_m\}$, we may assume $a_m\to a\in \R$ as $m\to\infty$. Since $d_m\to\infty$, the rescaled domain $\Omega_m$ approaches $\R\backslash\{a\}$. It follows from (2.22) and (2.24) that  
\begin{align*}
\sup_{x\in \Omega_{m}}d_{\Omega_m}^{1-\alpha}(x)\big|\nabla \Psi_{m}(x)\big|\leq 1,\\
\big|\nabla \Psi_{m}(0)\big|\geq \frac{1}{2}.
\end{align*}
Since $\Psi_m$ is harmonic in $\Omega_m$, there exists a subsequence of $\{\Psi_m\}$ (still denoted by $\{\Psi_m\}$) such that $\nabla \Psi_{m}\to \nabla \Psi$ locally uniformly in $\R\backslash \{a\}$. Then, the limit $\Psi$ is harmonic in $\R\backslash \{a\}$ and satisfies 
\begin{equation*}
\sup\Big\{|x-a|^{1-\alpha}\big|\nabla\Psi(x)\big| \bigm| x\in \R,\ x\neq a    \Big\}\leq 1.
\end{equation*}
We apply Proposition 2.8 and observe that $\Psi$ is harmonic at $x=a$. By applying the Liouville theorem, we conclude that $\nabla \Psi\equiv 0$. This contradicts $|\nabla\Psi(0)|\geq 1/2$ so Case 2 does not occur.

We reached a contradiction. The proof is now complete. 
\end{proof}

\vspace{10pt}

\section{A scale invariant H\"older-type estimate for the Helmholtz projection}

\vspace{10pt}

The goal of this section is to prove the H\"older-type estimate (1.12) (Lemma 3.3). We divide $\nabla \Phi=\q_{\Omega}\partial f$ into two terms $\nabla \Phi_{1}=\q_{\R}\partial f$ and $\nabla\Phi_{2}$. We estimate $\nabla \Phi_1$ by the Newton potential and $\nabla \Phi_{2}$ by the a priori estimate (2.8). In what follows, we identify $g\in C^{\infty}_{c}(\Omega)$ and its zero extension to $\R\backslash \overline{\Omega}$. Let $E$ denote the fundamental solution of the Laplace equation, i.e., $E(x)=C_{n}/ |x|^{n-2}$ for $n\geq 3$ and $E(x)=-1/(2\pi)\log{|x|}$ for $n=2$, where $C_n=(an(n-2))^{-1}$ and the volume of $n$-dimensional unit ball $a$.

\vspace{5pt}
 
\begin{prop}
For $g\in C^{\infty}_{c}(\R)$, set $h=E*g$. Then, $h\in C^{\infty}(\R)$ satisfies $\nabla^{2}h\in L^{2}(\R)$ and $-\Delta\ h=g$ in $\R$. Moreover, $-\nabla \D\ h$ agrees with $\q_{\R} g$ and the estimate 
\begin{equation*}
\Big[\nabla^{2}h \Big]^{(\alpha)}_{\R}+\Big[\q_{\R} g \Big]^{(\alpha)}_{\R}
\leq C_{\alpha}\Big[g\Big]^{(\alpha)}_{\R}    \tag{3.1} 
\end{equation*}
holds for $\alpha\in (0,1)$ with the dilation invariant constant $C_{\alpha}$, independent of $g$.
\end{prop}

\begin{proof}
We observe that $h\in C^{\infty}(\mathbb{R}^{n})$ satisfies $-\Delta h=g$ on $\R$. The second derivatives of $h$ are in $L^{2}$ since $||\nabla^{2}h||_{L^{2}}=||\Delta h||_{L^{2}}$. By pointwise kernel estimates of the fundamental solution, we have
\begin{align*}
\Big[\nabla^{2}h\Big]^{(\alpha)}_{\mathbb{R}^{n}}\leq C_{\alpha}\Big[g\Big]^{(\alpha)}_{\mathbb{R}^{n}}
\end{align*}
for $\alpha\in (0,1)$ with some constant $C_{\alpha}$ \cite[Lemma 4.4]{GT}. It remains to show that $\nabla \Psi=-\nabla \D h$ agrees with $\nabla \Phi=\q_{\R} g$. Since $\Phi$ and $\Psi$ satisfy the Poisson equation $\Delta \Phi=\D\ g$ in $\R$, $\tilde{\Phi}=\Phi-\Psi$ is harmonic and smooth in $\R$. By the mean-value formula, it follows that 
\begin{equation*}
\nabla \tilde{\Phi}(x)=\fint_{B_{x}(r)}\nabla \tilde{\Phi}(y)\dd y\quad \textrm{for}\ x\in \R,\ r>0.
\end{equation*}
Since $\nabla \tilde{\Phi}\in L^{2}(\R)$, applying the H\"older inequality yields 
\begin{equation*}
\big|\nabla \tilde{\Phi}(x)\big|\leq \frac{1}{a^{\frac{1}{2}}r^{\frac{n}{2}}}\big\|\nabla \tilde{\Phi}\big\|_{L^{2}(\R)}.
\end{equation*}
By sending $r\to\infty$, $\nabla \tilde{\Phi}\equiv 0$ follows. The proof is complete. 
\end{proof}

\vspace{5pt}

We next show that $\nabla \Phi=\q_{\Omega}g-\q_{\R}g$ solves the Neumann problem (2.1).

\vspace{5pt}

\begin{prop}
Let $\Omega$ be a uniformly $C^{1}$-domain in $\R$. Let $\alpha \in (0,1)$. Set $\nabla \Phi=\q_{\Omega}g-\q_{\R}g$ and $h=E*g$ for $g\in C^{\infty}_{c}(\Omega)$. Then, $\Phi$ is a solution of the Neumann problem (2.1) for 
\begin{equation*}
A=\nabla h-\nabla^{T}h.   
\end{equation*}
Moreover, the estimate 
\begin{equation*}
\sup_{x\in \Omega}d_{\Omega}^{1-\alpha}(x)\big|\nabla\Phi(x)\big|\leq C\Big[\nabla h-\nabla^{T}h\Big]^{(\alpha)}_{\Omega}    \tag{3.2}
\end{equation*}
holds provided that $\Omega$ is admissible for $\alpha$. The constant $C=C_{\alpha,\Omega}$ is invariant of dilation, translation, and rotation of $\Omega$.
\end{prop}

\begin{proof}
We first show that $\nabla\Phi=\q_{\Omega} g-\q_{\R}g$ satisfies the weak form (2.7) for $\varphi\in C^{1}_{c}(\overline{\Omega})$. By a mollification, we may assume $\varphi\in C^{\infty}_{c}(\overline{\Omega})$. We observe that   
\begin{align*}
\int_{\partial\Omega}An\cdot \nabla \varphi \dd{\cal{H}}
&=\int_{\partial\Omega}(\nabla h-\nabla^{T}h)n\cdot \nabla \varphi \dd{\cal{H}}\\
&=\sum_{i,j=1}^{n}\int_{\partial\Omega}(\partial_{j}h^{i}-\partial_{i}h^{j})n^{j}\partial_{i}\varphi \dd{\cal{H}}\\
&=\sum_{i,j=1}^{n}\int_{\Omega}\left((\partial_{j}^{2}h^{i}-\partial_{i}\partial_{j}h^{j})\partial_{i}\varphi+(\partial_{j}h^{i}-\partial_{i}h^{j})\partial_{i}\partial_{j}\varphi \right) \dd x  \\
&=(\Delta h-\nabla\D h,\nabla \varphi).
\end{align*}
Here, $(\cdot,\cdot)$ denotes the inner product on $L^{2}(\Omega)$. Since $-\Delta\ h=g$ in $\R$ and $\q_{\R}g=-\nabla\D\ h$ by Proposition 3.1, it follows that
\begin{align*}
\big(\nabla \Phi, \nabla\varphi \big)
&=\big(\q_{\Omega}g,\nabla\varphi\big)-\big(\q_{\R}g,\nabla\varphi\big)   \\
&=\big(g,\nabla\varphi\big)+\big(\nabla \D\ h, \nabla\varphi\big)   \\
&=\big(-\Delta h+\nabla \D\ h,\nabla \varphi\big).
\end{align*} 
Thus, $\nabla \Phi$ satisfies (2.7) for $A=\nabla h-\nabla^{T} h$ and $h=E*g$.

We next show 
\begin{equation*}
\sup_{x\in\Omega}d^{1-\alpha}_{\Omega}(x)\big|\nabla \Phi(x)\big|<\infty.   \tag{3.3}
\end{equation*}
Since $\q_{\Omega}$ acts as a bounded operator on $L^{p}\cap L^{2}(\Omega)$ for $2\leq p <\infty$ in a uniformly $C^{1}$-domain  \cite{FKS1}, \cite{FKS2}, $\nabla \Phi=\q_{\Omega}g-\q_{\R}g\in L^{p}\cap L^{2}(\Omega)$. Since $\Phi$ is harmonic in $\Omega$, it follows from the mean-value formula that
\begin{equation*}
\nabla \Phi(x)=\fint_{B_{x}(r)}\nabla \Phi(y)\dd y\quad  \textrm{for}\ x\in \Omega,\ r=\frac{d_{\Omega}(x)}{2}.
\end{equation*} 
Applying the H\"older inequality for $p\in [1,\infty)$ implies
\begin{equation*}
\big|\nabla \Phi(x)\big|\leq |B_{x}(r)|^{-\frac{1}{p}} \big\|\nabla \Phi\big\|_{L^{p}(\Omega)}.
\end{equation*}
Since $r=d_{\Omega}(x)/2$, it follows that
\begin{equation*}
\sup_{x\in\Omega}d^{\frac{n}{p}}_{\Omega}(x)\big|\nabla \Phi(x)\big|
\leq C_{p} \big\|\nabla\Phi\big\|_{L^{p}(\Omega)},
\end{equation*}
with the constant $C_{p}=2^{n/p}a^{-1/p}$. We take $p=n/(1-\alpha)> n$ so that $n/p=1-\alpha$. Then, (3.3) follows. Thus, $\Phi$ is a solution of (2.1) for $A=\nabla h-\nabla^{T}h$. The estimate (3.2) follows from the a priori estimate (2.8) with the dilation invariant constant $C=C_{\alpha,\Omega}$ by Remarks 2.4 (i). The proof is complete. 
\end{proof}

\vspace{5pt}

Propositions 3.1 and 3.2 now imply:

\vspace{5pt}

\begin{lem}
Let $\Omega$ be a strongly admissible, uniformly $C^{1}$-domain. Set $\nabla\Phi_{1}=\q_{\R}\partial f$ and   
$\nabla \Phi_{2}=\q_{\Omega}\partial f-\q_{\R}\partial f$ for $f\in C^{\infty}_{c}(\Omega)$. Then, the estimate
\begin{equation*}
\Big[\Phi_{1}\Big]^{(\alpha)}_{\R}
+\sup_{x\in \Omega}d^{1-\alpha}_{\Omega}(x)\big|\nabla \Phi_{2}(x)\big|
\leq C\Big[f\Big]^{(\alpha)}_{\Omega}    \tag{3.4}
\end{equation*}
holds for $\alpha\in (0,1)$. The constant $C=C_{\alpha,\Omega}$ is invariant of dilation, translation and rotation of $\Omega$.
\end{lem}

\begin{proof}
By Proposition 3.1, $\nabla\Phi_{1}=\q_{\R}\partial f$ agrees with $-\nabla \D h$ for $h=E*\partial f$ and $f\in C^{\infty}_{c}(\Omega)$. Since $E*g$ satisfies $[\nabla^{2}E*g]^{(\alpha)}_{\mathbb{R}^{n}}\leq C_{\alpha}[g]^{(\alpha)}_{\R}$ for $g\in C_{c}^{\infty}$ by (3.1), $h=\partial E*f$ satisfies 
\begin{equation*}
\Big[\nabla h\Big]^{(\alpha)}_{\R}
\leq C_{\alpha}'\Big[f\Big]^{(\alpha)}_{\R}   
=C_{\alpha}'\Big[f\Big]^{(\alpha)}_{\Omega}.    \tag{3.5}
\end{equation*}
Since $\Phi_{1}$ agrees with $-\D h$ up to an additive constant, by (3.5) we have 
\begin{equation*}
\Big[\Phi_{1}\Big]^{(\alpha)}_{\R}
=\Big[\D h\Big]^{(\alpha)}_{\R}
\leq C_{\alpha}'\Big[f\Big]^{(\alpha)}_{\Omega}.
\end{equation*}
We estimate $\nabla \Phi_2$ by using (3.2). Since $\Phi_2$ solves (2.1) and 
\begin{align*}
\Big[\nabla h-\nabla^{T}h  \Big]^{(\alpha)}_{\Omega}
&\leq \Big[\nabla h-\nabla^{T}h  \Big]^{(\alpha)}_{\R}   \\
&\leq 2\Big[\nabla h\Big]^{(\alpha)}_{\R}  
\leq 2C_{\alpha}'\Big[f\Big]^{(\alpha)}_{\Omega}
\end{align*}
by (3.5), we apply (3.2) and obtain the estimate for $\nabla \Phi_{2}$ in (3.4). Since the  constants $C_{\alpha}'$ and the constant in (3.2) are invariant of dilation of $\Omega$, so is $C=C_{\alpha,\Omega}$. The proof is complete.
\end{proof}

\vspace{5pt}

\section{Local H\"older estimates for the Stokes equations}

\vspace{5pt}

In this section, we review the local H\"older estimates for the Stokes equations (Lemma 4.3). We recall the $\tilde{L}^{p}$-theory for the Stokes equations in a general uniformly $C^{2}$-domain.\\

Let $\Omega$ be a domain in $\R$, $n\geq 2$. We define the space $\tilde{L}^{p}(\Omega)$ by 
\begin{equation*}
\tilde{L}^{p}(\Omega)=L^{p}(\Omega)\cap L^{2}(\Omega)\quad \textrm{ for}\ 2\leq p<\infty.
\end{equation*}
The space $\tilde{L}^{p}(\Omega)$ is a Banach space equipped with the norm  
\begin{equation*}
\big\|f\big\|_{\tilde{L}^{p}(\Omega)}=\max\Big\{\big\|f\big\|_{L^{p}(\Omega)}, \big\|f\big\|_{L^{2}(\Omega)}  \Big\}.
\end{equation*}
Let $L^{p}_{\sigma}(\Omega)$ denote the $L^{p}$-closure of $C_{c,\sigma}^{\infty}(\Omega)$. We define $G^{p}(\Omega)=\big\{\nabla \varphi\in L^{p}(\Omega)\ |\ \varphi\in L^{p}_{\textrm{loc}}(\Omega)  \big\}$. We define $\tilde{L}^{p}_{\sigma}(\Omega)$ and $\tilde{G}^{p}(\Omega)$ by a similar way. It is proved in \cite{FKS1} that for each $f\in \tilde{L}^{p}$, there exists a unique decomposition $f=f_0+\nabla \varphi$ by $f_0\in \tilde{L}^{p}_{\sigma}$ and $\nabla\varphi\in \tilde{G}^{p}$ satisfying 
\begin{equation*}
\big\|f_0\big\|_{\tilde{L}^{p}(\Omega)}+\big\|\nabla \varphi\big\|_{\tilde{L}^{p}(\Omega)}\leq C\big\|f\big\|_{\tilde{L}^{p}(\Omega)}
\end{equation*}
for uniformly $C^{2}$-domains $\Omega$ in $\mathbb{R}^{3}$. Thus, the Helmholtz projection $\p_{\Omega}: f\longmapsto f_0$ and $\q_{\Omega}=I-\p_{\Omega}$ exist and are bounded on $\tilde{L}^{p}$. Moreover, it is proved that the Stokes semigroup $S(t)$ is an analytic semigroup on $\tilde{L}^{p}_{\sigma}$ for uniformly $C^{2}$-domains. The result is extended to the $n$-dimensional case for $n\geq 2$ in \cite{FKS2}, \cite{FKS3}. Thus, solutions of the Stokes equations for $v_0\in \tilde{L}^{p}_{\sigma}$ are given by the Stokes semigroup and the Helmholtz projection defined on $\tilde{L}^{p}$, i.e., $v=S(t)v_0$ and $\nabla q=\mathbb{Q}_{\Omega}\Delta v$. We call $(v,q)$ $\tilde{L}^{p}$-solutions.\\

We estimate H\"older norms of $\tilde{L}^{p}$-solutions $(v,q)$ by applying the a priori estimate (2.4) for $\nabla q$.

\vspace{5pt}

\begin{prop}
Let $\Omega$ be a uniformly $C^{2}$-domain in $\R$, $n\geq 2$. Let $(v, q)$ be an $\tilde{L}^{p}$-solution of (1.1)--(1.4) for $p\geq n$. Then, the pressure $q$ is a solution of the Neumann problem (2.1) for  
\begin{equation*}
A=-\nabla v+\nabla^{T} v.
\end{equation*}
Moreover, the estimate 
\begin{equation}
\sup_{x\in\Omega}d_{\Omega}(x)\big|\nabla q(x,t)\big|\leq C\big\|\nabla v-\nabla^{T} v\big\|_{L^{\infty}(\partial\Omega)}(t)   \tag{4.1}
\end{equation} 
holds for $t\in (0,T)$ provided that $\Omega$ is admissible for $\alpha=0$. The constant $C=C_{\Omega}$ is invariant of dilation, translation and rotation of $\Omega$. 
\end{prop}

\begin{proof}
Although the assertion is essentially proved in \cite{AG1}, we give a proof in order to make the paper self-contained. We take an arbitrary $\varphi\in C^{2}_{c}\big(\overline{\Omega}\big)$ satisfying $\partial\varphi/\partial n=0$ on $\partial\Omega$. By $\D\ v=0$ in $\Omega$ and $v\cdot n=0$ on $\partial\Omega$, it follows that 
\begin{align*}
\int_{\partial\Omega}An\cdot \nabla \varphi \dd{\cal{H}}
&=\int_{\partial\Omega}(-\nabla v+\nabla^{T}v)n\cdot \nabla \varphi \dd{\cal{H}}\\
&=\sum_{i,j=1}^{n}\int_{\partial\Omega}(-\partial_{j}v^{i}+\partial_{i}v^{j})n^{j}\partial_{i}\varphi \dd{\cal{H}}\\
&=\sum_{i,j=1}^{n}\int_{\Omega}\big((-\partial_{j}^{2}v^{i}+\partial_{i}\partial_{j}v^{j})\partial_{i}\varphi+(-\partial_{j}v^{i}+\partial_{i}v^{j})\partial_{i}\partial_{j}\varphi \big) \dd x  \\
&=-\int_{\Omega}\Delta v\cdot \nabla \varphi \dd x 
=-\int_{\Omega}(v_{t}+\nabla q)\cdot \nabla \varphi \dd x
=\int_{\Omega}q\Delta  \varphi \dd x.
\end{align*}
So the pressure $q$ satisfies (2.3). Since $q$ is harmonic in $\Omega$, by the same way as in the proof of Proposition 3.2, we estimate
\begin{equation*}
d^{\frac{n}{s}}_{\Omega}(x)\big|\nabla q(x)\big|\leq C_{s}\big\|\nabla q\big\|_{L^{s}(\Omega)}\quad \textrm{for}\ x\in \Omega,
\end{equation*}
and all $s\in [1,\infty)$, where the time-variable of $q=q(x,\cdot)$ is suppressed. Since 
\begin{align*}
d_{\Omega}(x)\big|\nabla q(x)\big|
&= d^{\frac{n}{s}}_{\Omega}(x)\big|\nabla q(x)\big|d^{1-\frac{n}{s}}_{\Omega}(x)\\
&\leq C_{s}\big\|\nabla q\big\|_{L^{s}(\Omega)}d^{1-\frac{n}{s}}_{\Omega}(x),
\end{align*} 
we take $s=p\geq n$ for $d_{\Omega}(x)\leq 1$ and $s=2$ for $d_{\Omega}(x)\geq 1$ to estimate 
\begin{equation*}
\sup_{x\in \Omega}d_{\Omega}(x)\big|\nabla q(x)\big|\leq \tilde{C}_{p}\big\|\nabla q\big\|_{\tilde{L}^{p}(\Omega)}
\end{equation*}
with $\tilde{C}_{p}=\max\{C_{p},C_{2}\}$. Since $\q_{\Omega}$ acts as a bounded operator on $\tilde{L}^{p}$, $\nabla q$ is in $\tilde{L}^{p}$. Thus, $q$ is a solution of (2.1) for $A=-\nabla v+\nabla^{T}v$. The estimate (4.1) follows from (2.4) with a dilation invariant constant $C=C_{\Omega}$. The proof is complete.
\end{proof}

\vspace{5pt}

\begin{rems}
\noindent
(i) For $\alpha\in (0,1)$ and $p\geq n/(1-\alpha)$, the pressure $q$ (defined on $\tilde{L}^{p}$) is a solution of the Neumann problem (2.1) of type $\alpha$ and the estimate 
\begin{equation*}
\sup_{x\in \Omega}d^{1-\alpha}_{\Omega}(x)\big|\nabla q(x,t)\big|\leq C_{\alpha}\Big[\nabla v-\nabla^{T} v\Big]^{(\alpha)}_{\Omega}(t) 
\end{equation*}  
holds provided that $\Omega$ is strongly admissible.

\noindent
(ii) The pressure estimate (4.1) is proved in \cite{AG1} for the original admissible domains. The property of admissible for $\alpha=0$ is only used for (4.1) in order to prove the a priori estimate (1.7); see Proposition 6.2. 
\end{rems}

\vspace{5pt}

We recall the following notation for H\"older semi-norms of space-time functions \cite{LSU}. Let $f=f(x,t)$ be a real-valued or an $\R$-valued function defined in $Q=\Omega\times(0,T]$. For $\mu\in (0,1)$, we set the H\"older semi-norms 
\begin{equation*}
\Big[f\Big]^{(\mu)}_{t,Q}=\sup_{x\in \Omega}\Big[f\Big]^{(\mu)}_{(0,T]}(x),\
\Big[f\Big]^{(\mu)}_{x,Q}=\sup_{t\in (0,T]}\Big[f\Big]^{(\mu)}_{\Omega}(t).
\end{equation*}
In the parabolic scale for $\mu\in (0,1)$, we set 
\begin{equation*}
\Big[f\Big]^{(\mu,\frac{\mu}{2})}_{Q}=\Big[f\Big]^{(\frac{\mu}{2})}_{t,Q}+\Big[f\Big]^{(\mu)}_{x,Q}.
\end{equation*}

We estimate local H\"older norms for solutions of the Stokes equations both interior and up to boundary of $\Omega$. In the interior of $\Omega$, $\nabla q$ is smooth for spatial variables  since $q$ is harmonic in $\Omega$. Moreover, $\nabla q$ is H\"older continuous for a time variable by (4.1). We thus estimate H\"older norms of $\partial_{t}v$ and $\nabla^{2}v$ by the parabolic regularity theory \cite{LSU}. A corresponding estimate up to the boundary is more involved. By combining (4.1) and the Schauder estimate for the Stokes equations \cite{Sl77}, \cite{Sl07}, we estimate H\"older norms of $\partial_t v$, $\nabla^{2}v$, $\nabla q$ up to the boundary. We estimate H\"older norms of $\partial_t v$, $\nabla^{2}v$, $\nabla q$ by  
\begin{equation*}
N_{\delta,T}=\sup_{\delta\leq t\leq T}\big\|N(v,q)\big\|_{L^{\infty}(\Omega)}(t)\quad \textrm{for}\ \delta>0. 
\end{equation*}

\vspace{5pt}

The following H\"older estimate is proved in \cite[Proposition 3.2, Theorem 3.4]{AG1}.  
  
\vspace{5pt}
  
\begin{lem}
Let $\Omega$ be an admissible, uniformly $C^3$-domain of type $\alpha,\beta,K$ in $\R$.\\
(i) (Interior H\"older estimates)
Take $\mu\in(0,1)$, $\delta >0$, $T>0$, $R>0$. Then, there exists a constant $C=C\bigl(\delta,R,d,\mu,T, C_{\Omega}\bigr)$ such that the a priori estimate 
\begin{equation}
\Big[\nabla^{2}v\Big]^{(\mu,\frac{\mu}{2})}_{Q'}+\Big[v_t\Big]^{(\mu,\frac{\mu}{2})}_{Q'}+\Big[\nabla q\Big]^{(\mu,\frac{\mu}{2})}_{Q'}\leq CN_{\delta,T}  \tag{4.2}
\end{equation}
holds for all $\tilde{L}^{p}$-solutions $(v, q)$ for $p> n$ in $Q'=B_{x_0}(R)\times(2\delta,T]$ provided that $B_{x_0}(R)\subset\Omega$ and $x_0\in\Omega$, $2\delta<T$, where $d$ denotes the distance from $B_{x_0}(R)$ to $\partial\Omega$ and $C_{\Omega}$ is the constant in (4.1). The constant $C$ decreases in $d$.

\noindent 
(ii) (Estimates near the boundary) There exists $R_0=R_{0}(\alpha,\beta,K)>0$ such that for any $\mu\in(0,1),\ \delta>0$, $T>0$ and $R\leq R_{0}$, there exists a constant 
\begin{equation*}
C=C\bigl(\delta,\mu,T,R,\alpha,\beta,K, C_{\Omega}\bigr)
\end{equation*}
such that (4.2) holds for all $\tilde{L}^{p}$-solutions $(v, q)$ for $p> n$ in $Q'=\Omega_{x_0,R}\times(2\delta,T]$ for $\Omega_{x_0,R}=B_{x_0}(R)\cap\Omega$ and $x_0\in\partial\Omega$, $2\delta<T$.
\end{lem}

\vspace{5pt}

\section{Uniqueness in a half space}

\vspace{5pt}

The goal of this section is to prove the uniqueness for the Stokes equations (1.1)--(1.4) in a half space (Theorem 5.1). The  uniqueness theorem on $L^{\infty}$ is known for continuous velocity at time zero \cite{Sl03}. However, a blow-up limit may not be continuous nor even bounded as $t\downarrow 0$. Thus, we need a stronger uniqueness theorem in order to apply it for a blow-up limit. We prove a uniqueness theorem under suitable sup-bounds for velocity and pressure gradient. 

\vspace{5pt}

\begin{thm}
Let $v\in C^{2,1}\big(\overline{\mathbb{R}^{n}_{+}}\times (0,T]\big)$ and $\nabla q\in C\big(\overline{\mathbb{R}^{n}_{+}}\times (0,T]\big)$ satisfy (1.1)--(1.3), 
\begin{align*}
&\sup_{0<t\leq T}t^{\gamma}\big\|N(v,q)\big\|_{L^{\infty}(\R_{+})}(t)<\infty,    \tag{5.1}  \\
&\sup\left\{t^{\gamma+\frac{1}{2}}x_{n}\big|\nabla q(x,t)\big|\ \Big|\ x\in \R_{+},\ 0<t\leq T  \right\}<\infty,    \tag{5.2}
\end{align*}
for some $\gamma\in [0,1/2)$. Assume that $(v, q)$ satisfies 
\begin{equation*}
\int_{0}^{T}\int_{\R_{+}}\big(v\cdot(\partial_{t}\varphi+\Delta \varphi)-\nabla q\cdot \varphi\big)\textrm{d}x\textrm{d}t =0,     \tag{5.3}
\end{equation*}
for all $\varphi\in C^{\infty}_{c}(\R_{+}\times [0,T))$. Then, $v\equiv 0$ and $\nabla q\equiv 0$.
\end{thm}

\vspace{5pt}

We prove Theorem 5.1 from the following stronger assertion. 

\vspace{5pt}

\begin{lem}
Let $v\in C^{2,1}\big(\overline{\mathbb{R}^{n}_{+}}\times (0,T]\big)$ and $\nabla q\in C\big(\overline{\mathbb{R}^{n}_{+}}\times (0,T]\big)$ satisfy (1.1)--(1.3), 
\begin{align*}
&\sup_{0<t\leq T}t^{\gamma}\|v\|_{L^{\infty}(\R_{+})}(t)<\infty,    \tag{5.4}  \\
&\sup\left\{ t^{\gamma+\frac{1}{2}}(x_n^{2}+t)^{\frac{1}{2}}\big|\nabla q(x,t)\big|\ \Big|\ x\in \R_{+},\ 0<t\leq T  \right\}<\infty,     \tag{5.5}  
\end{align*}
for $\gamma\in [0,1/2)$, and $\nabla v$ is bounded in $\R_{+}$ for $t\in (0,T)$. Assume that $(v, q)$ satisfies (5.3) for all $\varphi\in C^{\infty}_{c}(\R_{+}\times [0,T))$. Then, $v\equiv 0$ and $\nabla q\equiv 0$.
\end{lem}

The uniqueness of the Stokes equations on $L^{\infty}(\R_{+})$ was first proved by V. A. Solonnikov based on a duality argument \cite[Theorem 1.1]{Sl03}. However, the result was restricted to continuous velocity at time zero. Recently, the author proved some  uniqueness theorem without assuming continuity of velocity at time zero. In the sequel, we give a proof for Lemma 5.2 based on the proof in \cite{A2}.   \\

We sketch the proof of Lemma 5.2. An essential step is to prove $\partial_{\textrm{tan}} v\equiv 0$. Once we know $\partial_{\textrm{tan}}v\equiv 0$, then $\nabla q\equiv 0$ and $v\equiv 0$ easily follow. In fact, the divergence-free condition for the velocity implies 
\begin{equation*}
\frac{\partial v^{n}}{\partial x_n}=-\sum_{j=1}^{n-1}\frac{\partial v^{j}}{\partial x_j}=0.
\end{equation*}
By the Dirichlet boundary condition, $v^{n}\equiv 0$ and $\partial q/\partial x_n\equiv 0$ follow. Thus, $\nabla q=\nabla q(x',t)$ is independent of the $x_n$-variable. The condition (5.5) implies $\nabla q\equiv 0$. By the uniqueness of the heat equation, $v\equiv 0$ follows. (We give a short proof for the uniqueness of the heat equation under the bound (5.4) in Lemma B.2).

Let $C^{\infty}_{c,\sigma}\big(\R_{+}\times (0,T)\big)$ denote the space of all smooth solenoidal vector fields with compact support in $\R_{+}\times (0,T)$. We prove
\begin{equation*}
\int_{0}^{T}\int_{\R_{+}}\partial_{\textrm{tan}}v(x,t)\cdot f(x,t)\dd x\dd t=0    \tag{5.6}
\end{equation*}
for all $f\in C^{\infty}_{c,\sigma}\big(\R_{+}\times (0,T)\big)$. Then, $\partial_{\textrm{tan}}v\equiv 0$ follows from de Rham's theorem.

\vspace{5pt}

\begin{prop}
Let $u\in L^{\infty}(\R_{+})\cap C^{1}\big(\overline{\R_{+}}\big)$ satisfy $\D\ u=0$ in $\R_{+}$, $u=0$ on $\{x_n=0\}$ and 
\begin{equation*}
\int_{\R_{+}}u\cdot f\dd x=0
\quad \textrm{for all}\ f\in C^{\infty}_{c,\sigma}(\R_{+}).
\end{equation*}
Then, $u\equiv 0$. 
\end{prop}

\begin{proof}
By de Rham's theorem (e.g., \cite[Theorem 1.1]{SS}), there exists a function $\Phi\in C^{2}\big(\overline{\R_{+}}\big)$ such that $u=\nabla \Phi$. Since $\D\ u=0$ in $\R_{+}$ and $u^{n}=0$ on $\{x_n=0\}$, the function $\Phi$ is harmonic in $\R_{+}$ and $\partial\Phi/\partial x_n=0$ on $\{x_n=0\}$. We extend $\Phi$ to $\R$ by the even extension, i.e.,  
\begin{equation*}
\tilde{\Phi}(x',x_n)=
\begin{cases}
&\Phi(x',x_n)\quad\hspace{8pt} \textrm{for}\  x_n\geq 0,\\
&\Phi(x',-x_n)\quad \textrm{for}\ x_n< 0.
\end{cases}
\end{equation*}
Then, $\tilde{\Phi}\in C^{2}(\R)$ is harmonic in $\R$. We apply the Liouville theorem for $\nabla\tilde{\Phi}\in L^{\infty}(\R)$ and conclude that $\nabla\tilde{\Phi}$ is constant. Since $\nabla \Phi$ is vanishing on $\{x_n=0\}$, $\nabla\Phi\equiv 0$ follows.  
\end{proof}

\vspace{5pt}

In the sequel, we prove (5.6). We consider the dual problem,\\
\begin{align*}
-\partial_{t}\varphi-\Delta \varphi+\nabla \pi&=\partial_{\textrm{tan}}f\hspace{10pt} \textrm{in}\quad \mathbb{R}^{n}_{+}\times (0,T),   \tag{5.7}\\
\textrm{div}\ \varphi &=0\hspace{28pt}\textrm{in}\quad \R_{+}\times(0,T), \tag{5.8}\\
\varphi& =0\hspace{28pt}\textrm{on}\quad  \partial\R_{+}\times(0,T), \tag{5.9} \\
\varphi& =0\hspace{28pt}\textrm{on}\quad \R_{+}\times\{t=T\}. \tag{5.10}
\end{align*}
It is proved in \cite{A2} (see Proposition A.1) that solutions $(\varphi,  \pi)$ exist and satisfy $\varphi\in S$ and $\nabla \pi\in L^{\infty}\big(0,T; L^{1}\big(\R_{+}\big)\big)$, where

\begin{equation*}
\begin{aligned}
S=\bigg\{& \varphi\in C^{\infty}\big(\overline{\R_{+}}\times [0,T]\big)\ \Big|\ \varphi, \nabla\varphi,\nabla^{2}\varphi, \partial_{t}\varphi, x_n^{-1}\varphi \in L^{\infty}\big(0,T; L^{1}\big(\R_{+}\big)\big),  \\
&  \varphi=0\ \textrm{on}\ \{x_n=0\}\cup\{t=T\}\ \bigg\}.    
\end{aligned}
\end{equation*}
It is noted that the solution $\varphi$ is in $L^{\infty}(0,T; L^{1})$ by $\nabla S(t)f\in L^{1}$ \cite{GMS} (\cite[Proposition 3.1]{A2}) although $S(t)f\notin L^{1}$ for general $f$ (i.e., there exists some $f\in L^{2}_{\sigma}\cap L^{1}$ such that $S(t)f\notin L^{1}$  \cite{DHP}, \cite{Sa}).\\

We complete the proof of Lemma 5.2 and then give a proof for the following Proposition 5.4 later in Appendix A.

 \vspace{5pt}
 
\begin{prop}
Under the assumption of Lemma 5.2, the initial condition (5.3) is extendable for all $\varphi \in {S}$.
\end{prop}

\begin{proof}[Proof of Lemma 5.2]
For $f\in C^{\infty}_{c,\sigma}\big(\R_{+}\times (0,T)\big)$, there exists a smooth solution $(\varphi,\pi)$ for (5.7)--(5.10) satisfying $\varphi\in S$ and $\nabla\pi\in L^{\infty}(0,T; L^{1})$ by Proposition A.1. Since the condition (5.3) is extendable for all test functions in $S$ by Proposition 5.4, it follows that   
\begin{equation*} 
\begin{aligned}
\int_{0}^{T}\int_{\R_{+}}v\cdot \partial_{\textrm{tan}}f \dd x\dd t
&=\int_{0}^{T}\int_{\R_{+}}v\cdot (-\partial_{t} \varphi-\Delta \varphi+\nabla \pi) \dd x\dd t   \\
&=\int_{0}^{T}\int_{\R_{+}}(v\cdot \nabla \pi-\nabla q\cdot \varphi)\dd x\dd t.   
\end{aligned}
\end{equation*}
Since $v(\cdot, t)\in L^{\infty}(\R_{+})\cap C^{2}(\overline{\R_{+}})$ satisfies $ \D\ v=0$ in $\R_{+}$, $v^{n}=0$ on $\{x_n=0\}$ and $\nabla \pi(\cdot,t)\in L^{1}(\R_{+})$, the first term vanishes (see, e.g.,  \cite[Proposition 2.3]{A2}). Similarly, the second term vanishes. Thus, (5.6) holds for all $f\in C^{\infty}_{c,\sigma}(\R_{+}\times (0,T))$. 

We apply Proposition 5.3 for $u=\partial_{\textrm{tan}}v$ and conclude that $\partial_{\textrm{tan}}v\equiv 0$. So $\nabla q\equiv 0$ and $v\equiv 0$. The proof is complete. 
\end{proof}

\begin{proof}[Proof of Theorem 5.1]
Since $(x_n^{2}+t)^{1/2}\leq x_n+t^{1/2}$, the condition (5.5) is satisfied under the assumptions (5.1) and (5.2). The assertion follows from Lemma 5.2.
\end{proof}

\vspace{5pt}

\section{A priori estimates for the Stokes flow }

\vspace{5pt}

We prove Theorem 1.3. We first show that $\tilde{L}^{p}$-solutions are sufficiently regular near $t=0$ and then prove (1.7) by a blow-up argument. After the proof of (1.7), we prove the estimate (1.8) by approximation.

\vspace{5pt}

\subsection{\rm{Regularity of $\tilde{L}^{p}$-solutions}}\mbox{}\\

\begin{prop}
Let $\Omega$ be a uniformly $C^{3}$-domain in $\R$, $n\geq 2$. Let $p>n$. Then, $\tilde{L}^{p}$-solutions $(v, q)$ for $v_0\in \tilde{L}^{p}_{\sigma}(\Omega)$ are bounded and H\"older continuous in $\overline{\Omega}\times [\delta,T]$ for each $\delta >0$ up to second orders. Moreover, for $\gamma \in (0,1)$, 
\begin{align*}
t^{\gamma}\big\|N(v,q)\big\|_{\infty}(t)\in C[0,T] \quad \textrm{and}\quad     
\lim_{t\downarrow 0}t^{\gamma}\big\|N(v,q)\big\|_{\infty}(t)=0,  \tag{6.1} 
\end{align*}
provided that $p>n/({2\gamma})$.
\end{prop}

\begin{proof}
We set 
\begin{equation*}
\tilde{N}(v,q)(x,t)=N(v,q)(x,t)+t^{\frac{3}{2}}\big|\nabla \partial_{t}v(x,t) \big|
+t^{\frac{3}{2}}\big|\nabla^{3} v(x,t) \big|+t^{\frac{3}{2}}\big|\nabla^{2} q(x,t) \big|.
\end{equation*}
We shall show that  
\begin{equation*}
\sup_{0\leq t\leq T}\big\|\tilde{N}(v,q) \big\|_{L^{p}_{\textrm{ul}}(\Omega)}(t)\leq C\big\|v_0 \big\|_{\tilde{L}^{p}(\Omega)}  \tag{6.2}
\end{equation*}
for $\tilde{L}^{p}$-solutions $(v,q)$ for $p>n$, where $L^{p}_{\textrm{ul}}(\Omega)$ denotes the uniformly local $L^{p}$ space and is equipped with the norm
\begin{equation*}
\big\|f\big\|_{{L}^{p}_{\textrm{ul},r}(\Omega)}=\sup_{x_0\in \Omega}\big\|f\big\|_{L^{p}(\Omega_{x_0,r})},\quad \Omega_{x_0,r}=B_{x_0}(r)\cap \Omega. 
\end{equation*}
We define the space $W^{1,p}_{\textrm{ul}}(\Omega)$ by a similar way. For simplicity, we suppress the subscript for $r=1$, i.e., $||f||_{L^{p}_{\textrm{ul},1}}=||f||_{L^{p}_{\textrm{ul}}}$ and $\Omega_{x_0,1}=\Omega_{x_0}$. 

We observe from (6.2) that $(v,q)$ is bounded and H\"older continuous in $\overline{\Omega}\times [\delta,T]$ for $\delta >0$. In fact, by the Sobolev embedding we estimate
\begin{equation*}
\sup_{\delta \leq t\leq T}\big\|N(v,q) \big\|_{L^{\infty}(\Omega)}(t)\leq C\big\|v_0 \big\|_{\tilde{L}^{p}(\Omega)}.  \tag{6.3}
\end{equation*}  
Thus, $(v,q)$ is bounded in $\Omega\times [\delta ,T]$. Moreover, $\partial_{t}v$, $\nabla^{2} v$, $\nabla q$ are H\"older continuous in $\overline{\Omega}$ for each $t\in (0,T]$. We observe that $\nabla^{2} v(\cdot,t)$ is H\"older continuous in $[\delta, T]$. Let $A$ denote the generator of $S(t)$ on $\tilde{L}^{p}_{\sigma}$ and $D(A)$ denote the domain of $A$ in $\tilde{L}^{p}_{\sigma}$. We may assume $v_0\in D(A)$. For $t>s\geq \delta $, it follows from (6.3) that 
\begin{align*}
\big\|\nabla^{2}v(t)-\nabla^{2}v(s)\big\|_{L^{\infty}(\Omega)}
&\leq \int_{s}^{t}\big\|\nabla^{2}S(r)Av_0\big\|_{L^{\infty}(\Omega)}\dd r   \\
&\leq C|t-s|\big\|Av_0\big\|_{\tilde{L}^{p}(\Omega)}.
\end{align*}
Thus, $\nabla^{2} v(\cdot,t)$ is H\"older continuous in $[\delta ,T]$. By a similar way, we are able to prove that $\partial_{t}v$ and $\nabla q$ are H\"older continuous in $[\delta ,T]$. We proved that $(v,q)$ is bounded and H\"older continuous in $\overline{\Omega}\times [\delta,T]$. In particular, $||N(v,q)||_{\infty}(t)\in C(0,T]$.

We prove (6.1) by applying the interpolation inequality,
\begin{equation*}
\big\|\varphi\big\|_{L^{\infty}(\Omega)}\leq \frac{C}{r^{\frac{n}{p}}}
\left(\big\|\varphi\big\|_{L^{p}_{\textrm{ul},r}(\Omega)}+r\big\|\nabla \varphi\big\|_{L^{p}_{\textrm{ul},r}(\Omega)} \right)   \tag{6.4}
\end{equation*}
for $\varphi\in W^{1,p}_{\textrm{ul}}\big(\Omega)$ and $r\leq r_0$ (see \cite[Lemma 3.1.4]{L}). We may assume $r_0\leq 1$. We substitute $\varphi=v$ and $r=t^{1/2}$ into (6.4) to estimate  
\begin{align*}
\big\|v\big\|_{L^{\infty}(\Omega)}
&\leq \frac{C}{t^{\frac{n}{2p}}}\Big(\big\|v\big\|_{L^{p}_{\textrm{ul}}(\Omega)}+t^{\frac{1}{2}}\big\|\nabla v\big\|_{L^{p}_{\textrm{ul}}(\Omega)} \Big) \\
&\leq \frac{C'}{t^{\frac{n}{2p}}}\big\|v_0\big\|_{\tilde{L}^{p}(\Omega)}
\end{align*}
by (6.2). By a similar way, we apply (6.4) for $\nabla v$, $\nabla^{2} v$, $\partial_{t} v$, $\nabla q$ and observe that 
\begin{equation*}
\sup_{0<t\leq 1}t^{\frac{n}{2p}}\big\|N(v,q)\big\|_{\infty}(t)<\infty.
\end{equation*}
Thus, $t^{\gamma}||N(v,q)||_{\infty}(t)$ is continuous in $[0,T]$ and takes zero at $t=0$ provided that $p>n/(2\gamma)$.

It remains to show (6.2). By estimates of $S(t)$ and $\p$ on $\tilde{L}^{p}$ \cite[Theorem 1.3]{FKS3}, it follows that 
\begin{equation*}
\sup\limits_{0\leq t\leq {T}}\bigl\|N(v,q)\bigr\|_{\tilde{L}^{p}(\Omega)}(t)\leq C\big\|v_0\big\|_{\tilde{L}^{p}(\Omega)}.     \tag{6.5}
\end{equation*}
Moreover, we have 
\begin{equation*}
\sup\limits_{0\leq t\leq T}t^{\frac{3}{2}}\big\|\nabla \partial_{t}v\big\|_{\tilde{L}^{p}(\Omega)}(t)\leq C\big\|v_0\big\|_{\tilde{L}^{p}(\Omega)},        \tag{6.6}
\end{equation*}
since $\partial_{t}v={A}e^{t{A}}v_0=e^{\frac{t}{2}{A}}Ae^{\frac{t}{2}{A}}v_0$. We estimate the uniformly local $L^{p}$-norms of $\nabla^{3}v$ and $\nabla^{2}q$. For $x_0\in \Omega$, we take a $C^{3}$-bounded domain $\Omega'$ such that $\Omega_{x_0}\subset \Omega'\subset \Omega_{x_0,2}$ and set the average of $q$ in $\Omega'$ by 
\begin{equation*}
\big(q\big)=\fint_{\Omega'}q \dd x.
\end{equation*}
By the Poincar\'e inequality \cite[5.8.1]{E}, we estimate 
\begin{equation*}
\big\|q-(q)\big\|_{L^{p}(\Omega')}\leq C\big\|\nabla q\big\|_{L^{p}(\Omega')}.  \tag{6.7}
\end{equation*}
The constant $C$ depends on $x_0\in \Omega$ and the boundary regularity of $\Omega$, but is uniformly bounded for $x_0\in\Omega$ since $\Omega'\subset B_{x_0}(2)$ and the boundary $\partial\Omega$ is uniformly regular. We shift $q$ to $\hat{q}=q-(q)$. By the higher-order regularity theory \cite[Chapter IV.4 and 5]{Gal} for the stationary Stokes equations (for each $t>0$), 
\begin{align*}
-\Delta v+\nabla \hat{q}
&=-\partial_{t}v\quad \textrm{in}\ \Omega, \\
\D\ v
&=0\qquad\hspace{4pt} \textrm{in}\ \Omega,\\
v
&=0\qquad\hspace{4pt} \textrm{on}\ \partial\Omega,
\end{align*}
we estimate
\begin{equation*}
\big\|\nabla^{3}v\big\|_{L^{p}(\Omega_{x_0})}+\big\|\nabla^{2} q\big\|_{L^{p}(\Omega_{x_0})}
\leq C\left(\big\|\partial_{t}v\big\|_{W^{1,p}(\Omega')}+\big\|v\big\|_{W^{1,p}(\Omega')}+\big\|\hat{q}\big\|_{L^{p}(\Omega')}   \right),   \tag{6.8}
\end{equation*}
with some constant $C$ independent of $x_0\in\Omega$. Since $x_0\in \Omega$ is an arbitrary point, by (6.5)--(6.8) we obtain  
\begin{equation*}
\sup_{0\leq t\leq T}\Big( t^{\frac{3}{2}}\big\|\nabla^{3}v\big\|_{L^{p}_{\textrm{ul}}(\Omega)}(t)+t^{\frac{3}{2}}\big\|\nabla^{2}q\big\|_{L^{p}_{\textrm{ul}}(\Omega)}(t)\Big)
\leq C\big\|v_0\big\|_{\tilde{L}^{p}(\Omega)}.     
\end{equation*}
We proved (6.2). The proof is complete.
\end{proof}

\vspace{5pt}

\subsection{\rm{A blow-up argument}}\mbox{}\\

Now, we prove the a priori estimate (1.7) by a blow-up argument. For $\alpha\in (0,1)$ we set 
\begin{equation*}
\gamma =\frac{1-\alpha}{2}.
\end{equation*}
Then, $t^{\gamma}||N(v,q)||_{\infty}(t)$ is continuous in $[0,T]$ and takes zero at $t=0$ for $\tilde{L}^{p}$-solutions $(v,q)$ for $v_0\in \tilde{L}^{p}_{\sigma}$ provided that $p> n/(2\gamma)$ by (6.1).

\vspace{5pt}

\begin{prop}
Let $\Omega$ be a strongly admissible, uniformly $C^{3}$-domain. For $\alpha\in (0,1)$ and $p>n/(1-\alpha)$, there exist some constants $T_0$ and $C$ such that (1.7) holds for all $\tilde{L}^{p}$-solutions for $v_0=\p\partial f, f\in C^{\infty}_{c}(\Omega)$.
\end{prop}

\begin{proof}
We argue by contradiction. Suppose on the contrary that (1.7) were false for any choice of constants $C$ and $T_0$. Then, there would exist a sequence of $\tilde{L}^{p}$-solutions $(v_m,q_m)$ for $v_{0,m}=\p \partial f_{m}$, $f_m\in C^{\infty}_{c}(\Omega)$ such that 
\begin{equation*}
\sup_{0\leq t\leq 1/m}t^{\gamma}\big\|N(v_m,q_m)\big\|_{\infty}(t)>m\Big[f_m\Big]^{(\alpha)}_{\Omega}.
\end{equation*}   
We take a point $t_m\in (0,1/m)$ such that 
\begin{equation*}
t_m^{\gamma}\big\|N(v_m,q_m)\big\|_{\infty}(t_m)\geq \frac{1}{2}M_{m},\quad M_{m}=\sup_{0\leq t\leq 1/m}t^{\gamma}\big\|N(v_m,q_m)\big\|_{\infty}(t),
\end{equation*}   
and divide $(v_m,q_m)$ by $M_{m}$ to get $\tilde{v}_{m}=v_m/M_{m}, \tilde{q}_{m}=q_m/M_{m}$ and $\tilde{f}_{m}=f_m/M_m$ satisfying 
\begin{align*}
\sup_{0\leq t\leq t_m}t^{\gamma}\big\|N(\tilde{v}_m, \tilde{q}_m)\big\|_{\infty}(t)
&\leq 1,\\
\Big[\tilde{f}_m\Big]^{(\alpha)}_{\Omega}
&< \frac{1}{m},\\
t_m^{\gamma}\big\|N(\tilde{v}_m,\tilde{q}_m)\big\|_{\infty}(t_m)
&\geq \frac{1}{2}.
\end{align*}
We take a point $x_m\in \Omega$ such that 
\begin{equation*}
t_m^{\gamma}N(\tilde{v}_m,\tilde{q}_m)(x_m,t_m)\geq \frac{1}{4},
\end{equation*}
and rescale $(\tilde{v}_m,\tilde{q}_m)$ around $(x_m,t_m)$ to get a blow-up sequence 
\begin{equation*}
u_m(x,t)=t^{\gamma}_{m}\tilde{v}_m(x_m + t_m^{\frac{1}{2}}x,t_m t),
\quad p_m(x,t)=t_m^{\gamma+\frac{1}{2}}\tilde{q}_m(x_m + t_m^{\frac{1}{2}}x,t_m t),
\end{equation*}
and 
\begin{equation*}
g_{m}(x)=t_{m}^{-\frac{\alpha}{2}}\tilde{f}_{m}(x_m+t_{m}^{\frac{1}{2}}x).
\end{equation*}
The blow-up sequence $(u_m,p_m)$ satisfies (1.1)--(1.4) in $\Omega_{m}\times (0,1]$ for $u_{0,m}=\p_{\Omega_{m}}\partial g_{m}$ and  
\begin{equation*}
\Omega_{m}=\frac{\Omega-\{x_m\}}{t^{\frac{1}{2}}_{m}}.
\end{equation*}
The estimates for $(\tilde{v}_m,\tilde{q}_m)$ are inherited to   
\begin{align*}
\sup_{0\leq t\leq 1}t^{\gamma}\big\|N(u_m, p_m)\big\|_{L^{\infty}(\Omega_{m})}(t)
&\leq 1,             \tag{6.9}  \\
\Big[g_m\Big]^{(\alpha)}_{\Omega_m}
&< \frac{1}{m},          \tag{6.10}   \\
N(u_m,p_m)(0,1)
&\geq \frac{1}{4}.       \tag{6.11}
\end{align*}
We set $c_m={d_{m}}/{t_{m}^{\frac{1}{2}}}$ for $d_{m}=d_{\Omega}(x_m)$. Then, the proof is divided into two cases depending on whether $\{c_m\}$ converges or not.\\

\noindent 
\textit{Case} 1 $\overline{\lim}_{m\to\infty}c_m=\infty$. We may assume ${\lim}_{m\to\infty}c_m=\infty$. In this case, the rescaled domain $\Omega_{m}$ expands to the whole space. In fact, for each $R>0$ we observe that 
\begin{equation*}
\inf \big\{d_{\Omega_{m}}(x) \bigm| |x|\leq R \big\}\to \infty
\quad \textrm{as}\ c_m\to\infty.
\end{equation*}
We take an arbitrary $\varphi\in C_{c}^{\infty}(\R\times [0,1))$. We may assume that $\varphi$ is supported in $\Omega_{m}\times [0,T)$. Since $(u_m,p_m)$ satisfies (1.1) in $\Omega_{m}\times (0,1]$ for $u_{0,m}=\partial g_m-\nabla \Phi_{0,m}$ and $\nabla \Phi_{0,m}=\q_{\Omega_{m}}\partial g_m$, it follows that   
\begin{equation*}
\int_{0}^{1}\int_{\R}\big(u_{m}\cdot (\partial_{t}\varphi+\Delta \varphi)-\nabla p_{m}\cdot \varphi\big)\dd x\dd t
=\int_{\Omega_m}(g_m\cdot \partial\varphi_{0}-\Phi_{0,m}\D\ \varphi_{0})\dd x,    \tag{6.12}
\end{equation*}
where $\varphi_{0}(x)=\varphi(x,0)$.

We apply Lemma 4.3 (i) and observe that $\partial_{t}u_m$, $\nabla^{2}u_m$, $\nabla p_m$ are equi-continuous in the interior of $\Omega_{m}$. There exists a subsequence of $\{(u_m,p_m)\}$ (still denoted by $\{(u_m,p_m)\}$) such that $(u_m,p_m)$ converges to a limit $(u,p)$ locally uniformly in $\R\times (0,1]$ together with $\nabla u_m$, $\nabla^{2}u_m$, $\partial_t u_m$, $\nabla p_m$. Moreover, it follows from (4.1) and (6.9) that  
\begin{equation*}
\sup\left\{t^{\gamma+\frac{1}{2}}d_{\Omega_{m}}(x)\big|\nabla p_{m}(x,t)\big|\ 
\Big|\ x \in \Omega_{m},\ 0<t\leq 1    \right\}\leq C,   \tag{6.13}
\end{equation*} 
with some constant $C$ independent of $m$. Since $\Omega_m$ expands to the whole space, $\nabla p_{m}$ converges to zero locally uniformly in $\R\times (0,1]$, i.e., $\nabla p\equiv 0$. 

We apply Lemma 3.3 for $\nabla \Phi_{1,m}=\q_{\R}\partial g_m$ and $\nabla \Phi_{2,m}=\q_{\Omega_m}\partial g_m-\q_{\R}\partial g_m$ to estimate    
\begin{equation*}
\Big[\Phi_{1,m}\Big]^{(\alpha)}_{\R}
+\sup_{x\in \Omega_{m}}d_{\Omega_{m}}^{1-\alpha}(x)\big|\nabla\Phi_{2,m}(x)\big|
\leq C\Big[g_m\Big]^{(\alpha)}_{\Omega_{m}},   \tag{6.14}
\end{equation*} 
with some constant $C$ independent of $m$. By (6.10) and (6.14), the right-hand side of (6.12) vanishes as $m\to\infty$. Thus, the limit $u$ satisfies 
\begin{equation*}
\int_{0}^{1}\int_{\R}u\cdot (\partial_{t}\varphi+\Delta \varphi)\dd x\dd t=0.
\end{equation*}
By the uniqueness of the heat equation (Proposition B.1), we conclude that $u\equiv 0$ (and $\nabla p\equiv 0$). This contradicts $N(u,p)(0,1)\geq 1/4$ by (6.11) so Case 1 does not occur.\\

\noindent 
\textit{Case} 2 $\overline{\lim}_{m\to\infty}c_m<\infty$. By choosing a subsequence, we may assume $\lim_{m\to\infty}c_m=c_0$ for some $c_0\geq 0$. In this case, the rescaled domain $\Omega_m$ expands to a half space. Since $d_{\Omega}(x_m)=c_m t_m^{1/2}\to 0$, the points $\{x_m\}$ accumulate to the boundary. By translation and rotation around $\tilde{x}_{m}\in \partial\Omega$, the projection of $x_m$ to $\partial\Omega$, we may assume that $x_m=(0,d_m)$ and $\tilde{x}_{m}=0$. We consider the neighborhood of the origin denoted by   
\begin{equation*}
\Omega_{\textrm{loc}}=\left\{(x',x_n)\in \R\ \big|\ h(x')<x_n<h(x')+\beta',\ |x'|<\alpha'\right\},
\end{equation*} 
with some constants $\alpha',\beta', K'$ and a $C^{3}$-function $h$ satisfying $h(0)=0$, $\nabla'h(0)=0$ and $||h||_{C^{3}(\{|x'|<\alpha'\})}\leq K'$. Since $\Omega_{\textrm{loc}}\subset \Omega$ is rescaled to 
\begin{equation*}
\Omega_{\textrm{loc},m}=\left\{(x',x_n)\in \R\ \Bigg|\ h_{m}(x')-c_m<x_n<h_{m}(x')-c_m+\frac{\beta'}{t_m^{\frac{1}{2}}},\ |x'|<\frac{\alpha'}{t_{m}^{\frac{1}{2}}}\right\},
\end{equation*} 
where $h_m(x')=t_{m}^{-1/2}h(t_m^{1/2}x')$, $\Omega_{\textrm{loc},m}$ expands to the half space $\R_{+,-c_0}=\{(x',x_n)\ |\ x_n>-c_0 \}$.

We take an arbitrary $\varphi\in C^{\infty}_{c}(\R_{+,-c_0}\times [0,1))$ and observe that $\varphi$ is supported in $\Omega_{\textrm{loc},m}\times [0,1)$ for sufficiently large $m$. Since $(u_m,p_m)$ satisfies (1.1), it follows that 
\begin{equation*}
\int_{0}^{1}\int_{\Omega_{\textrm{loc},m}}\big(u_{m}\cdot (\partial_{t}\varphi+\Delta \varphi)-\nabla p_{m}\cdot \varphi\big)\dd x\dd t
=\int_{\Omega_m}(g_m\cdot \partial\varphi_{0}-\Phi_{0,m}\D\ \varphi_{0})\dd x.  \tag{6.15}
\end{equation*}
We apply Lemma 4.3 (ii) and observe that $\partial_{t}u_m$, $\nabla^{2}u_m$, $\nabla p_m$ are equi-continuous up to the boundary of $\Omega_m$. There exists a subsequence denoted by $\{(u_m,p_m)\}$ such that $(u_m,p_m)$ converges to a limit $(u,p)$ locally uniformly in $\overline{\mathbb{R}^{n}_{+,-c_0}}\times (0,1]$ together with $\nabla u_m$, $\nabla^{2}u_m$, $\partial_t u_m$, $\nabla p_m$. By (6.13), the limit $p$ satisfies  
\begin{equation*}
\sup\left\{t^{\gamma+\frac{1}{2}}(x_n+c_0)\big|\nabla p(x,t)\big|\ 
\Big|\ x \in \R_{+,-c_0},\ 0<t\leq 1    \right\}\leq C.   
\end{equation*} 
By (6.10), (6.14) and sending $m\to\infty$, the right-hand side of (6.15) vanishes as in Case 1. Thus, the limit $(u,p)$ satisfies 
\begin{equation*}
\int_{0}^{1}\int_{\R_{+,-c_0}}\big(u\cdot (\partial_{t}\varphi+\Delta \varphi)-\nabla p\cdot \varphi\big)\dd x\dd t
=0.    
\end{equation*}
We apply Theorem 5.1 and conclude that $u\equiv 0$ and $\nabla p\equiv 0$. This contradicts $N(u,p)(0,1)\geq 1/4$ by (6.11) so Case 2 does not occur.

We reached a contradiction. The proof is now complete. 
\end{proof}

\vspace{5pt}

\subsection{\rm{Approximation}}\mbox{}\\

We prove the estimate (1.8) by interpolation and approximation. After the proof of Theorem 1.3, we give a proof for Theorems 1.2 and 1.1.

\vspace{5pt}

\begin{prop}
Let $\Omega$ be a domain in $\R$. Then, the estimate  
\begin{equation*}
\Big[f\Big]^{(\alpha)}_{\Omega}\leq 2\big\|f\big\|_{L^{\infty}(\Omega)}^{1-\alpha}\big\|\nabla f\big\|_{L^{\infty}(\Omega)}^{\alpha}     \tag{6.16}
\end{equation*}
holds for $f\in C^{\infty}_{c}(\Omega)$ and $\alpha\in (0,1)$. 
\end{prop}

\begin{proof}
We identify $f\in C^{\infty}_{c}(\Omega)$ and its zero extension to $\R\backslash \overline{\Omega}$. For arbitrary $x,y\in \R$, $x\neq y$, we estimate 
\begin{align*}
\frac{\big|{f}(x)-{f}(y)\big|}{|x-y|^{\alpha}}
&=\big|{f}(x)-{f}(y)\big|^{1-\alpha}\Bigg(\frac{\big|{f}(x)-{f}(y)\big|}{|x-y|}\Bigg)^{\alpha}  \\
&\leq 2\big\|{f}\big\|_{L^{\infty}(\R)}^{1-\alpha}\big\|\nabla {f}\big\|_{L^{\infty}(\R)}^{\alpha}.
\end{align*}
Since $f$ is supported in $\Omega$, (6.16) follows.
\end{proof}

\vspace{5pt}

\begin{prop}
Let $\Omega$ be a domain with Lipschitz boundary. \\
(i) When $\Omega$ is bounded, 
\begin{equation*}
C^{1}_{0}(\Omega)=\Big\{f\in C^{1}(\overline{\Omega})\ \big|\ f=0,\ \nabla f=0\ \textrm{on}\ \partial\Omega \Big\}.
\end{equation*}
(ii) When $\Omega$ is unbounded, 
\begin{equation*}
C^{1}_{0}(\Omega)=\Big\{f\in C^{1}(\overline{\Omega})\ \big|\ f\ \textrm{and}\ \nabla f\ \textrm{are vanishing on}\  \partial\Omega\ \textrm{and as}\ |x|\to\infty \Big\}.
\end{equation*}
Moreover, $C^{\infty}_{c}(\Omega)$ is dense in $C^{1}_{0}\cap W^{1,2}(\Omega)$.
\end{prop}

\begin{proof}
We begin with the case when $\Omega$ is star-shaped, i.e., $\lambda \Omega\subset \overline{\Omega}$ for $\lambda<1$. We take $f\in C^{1}(\overline{\Omega})$ satisfying $f=0$ and $\nabla f=0$ on $\partial\Omega$ and set 
\begin{align*}
f_{\lambda}(x)=
\begin{cases}
&f(x/\lambda)\quad \textrm{for}\ x\in \lambda\Omega, \\
&0\qquad\quad\ \textrm{for}\ x\in \Omega\backslash \overline{\lambda\Omega}.
\end{cases}
\end{align*}
Since $f$ and $\nabla f$ are vanishing on $\partial\Omega$, $f_{\lambda}$ is continuously differentiable in $\Omega$. Since $f_{\lambda}$ converges to $f$ in $W^{1,\infty}(\Omega)$ as $\lambda\to 1$, by mollification of $f_{\lambda}$, we obtain the sequence $\{f_m\}\subset C_{c}^{\infty}(\Omega)$ satisfying $f_m\to f$ in $W^{1,\infty}(\Omega)$. When $\Omega$ is a general bounded Lipschitz domain, we decompose $\Omega$ into star-shaped domains (see \cite[Lemma II.1.3]{Gal}) and reduce the problem to the case of star-shaped. When $\Omega$ is bounded,  $C^{\infty}_{c}$ is dense in $C^{1}_{0}\cap W^{1,2}=C_{0}^{1}$. We proved (i). 

We prove (ii). For $f\in C^{1}(\overline{\Omega})$ satisfying 
\begin{align*}
&f=0,\quad \nabla f=0\quad\textrm{on}\ \partial\Omega,\\
&\lim_{|x|\to\infty}f(x)=0,\quad \lim_{|x|\to\infty}\nabla f(x)=0,
\end{align*}
we prove that there exists a sequence $\{f_m\}\subset C^{\infty}_{c}$ such that 
\begin{equation*}
\lim_{m\to\infty}\big\|f-f_m\big\|_{W^{1,\infty}(\Omega)}=0.  \tag{6.17}
\end{equation*}
Here, we write $\lim_{|x|\to\infty}f(x)=0$ in the sense that $f(x_m)\to0$ as $m\to\infty$ for any sequence $\{x_m\}\subset \Omega$ such that $|x_m|\to\infty$. This condition is equivalent to 
\begin{equation*}
\lim_{R\to\infty}\sup\Big\{f(x)\ \big|\ x\in \Omega, \ |x|\geq R\Big\}=0.
\end{equation*}
Let $\theta\in C^{\infty}_{c}[0,\infty)$ be a cutoff function such that $\theta\equiv 1$ in $[0,1/2]$, $\theta\equiv 0$ in $[1,\infty)$ and $0\leq \theta\leq 1$. We set $\theta_{m}(x)=\theta(|x|/m)$ so that $\theta_{m}\in C^{\infty}_{c}(\R)$ satisfies $\theta_{m}\equiv 1$ for $|x|\leq m/2$ and $\theta_{m}\equiv 0$ for $|x|\geq m$. We observe that $\tilde{f}_{m}=f\theta_{m}$ satisfies $\tilde{f}_m\in C^{1}(\overline{\Omega})$ and spt $\tilde{f}_m\subset \overline{\Omega}\cap \{|x|\leq m \}$. Since $f$ and $\nabla f$ are vanishing on $\partial\Omega$, $\tilde{f}_m$ satisfies $\tilde{f}_m=0$ and $\nabla \tilde{f}_m=0$ on $\partial\Omega$. Moreover, $\tilde{f}_{m}$ converges to $f$ uniformly in $\overline{\Omega}$ since $f$ is decaying as $|x|\to\infty$, i.e., 
\begin{align*}
\big\|f-\tilde{f}_{m}\big\|_{L^{\infty}(\Omega)}
&=\big\|f(1-\theta_{m})\big\|_{L^{\infty}(\Omega)}  \\
&\leq \sup\big\{f(x)\ |\ x\in \Omega,\ |x|\geq m/2\   \big\}\\
&\to 0\quad \textrm{as}\ m\to\infty.
\end{align*}
By a similar way, $\nabla \tilde{f}_{m}$ converges to $\nabla f$ uniformly in $\overline{\Omega}$. Thus, we have 
\begin{equation*}
\lim_{m\to\infty}\big\|f-\tilde{f}_m\big\|_{W^{1,\infty}(\Omega)}=0.  
\end{equation*}
We set $\Omega_{m}=\Omega\cap B_{0}(m)$. We may assume that $\Omega_{m}$ has Lipschitz boundary by taking a bounded Lipschitz domain $\Omega'_{m}\supset \Omega_{m}$ if necessary. Since $\tilde{f}_{m}\in C^{1}_{0}(\Omega_{m})$ by the assertion (i), for each $m\geq 1$ there exists $\{\tilde{f}_{m,k}\}\subset C^{\infty}_{c}(\Omega_{m})$ such that 
\begin{equation*}
\lim_{k\to\infty}\big\|\tilde{f}_{m}-\tilde{f}_{m,k}\big\|_{W^{1,\infty}(\Omega_{m})}=0,
\end{equation*}
i.e., for an arbitrary $\varepsilon>0$ there exists $K=K_{m,\varepsilon}$ such that 
\begin{equation*}
\big\|\tilde{f}_{m}-\tilde{f}_{m,k}\big\|_{W^{1,\infty}(\Omega_{m})}
\leq \varepsilon\quad \textrm{for}\ k\geq K_{m,\varepsilon}.
\end{equation*}
We set $f_{m}=\tilde{f}_{m,k}$ for $k=K_{m,\varepsilon}$. Then, $f_m\in C^{\infty}_{c}(\Omega)$ satisfies 
\begin{align*}
\big\|f- f_{m}\big\|_{W^{1,\infty}(\Omega)}
&\leq \big\|f- \tilde{f}_{m}\big\|_{W^{1,\infty}(\Omega)}
+\big\|\tilde{f}_{m}-f_m  \big\|_{W^{1,\infty}(\Omega_{m})} \\
&\leq \big\|f- \tilde{f}_{m}\big\|_{W^{1,\infty}(\Omega)}+\varepsilon.
\end{align*}
It follows that 
\begin{equation*}
\overline{\lim_{m\to\infty}}\big\|f-f_m\big\|_{W^{1,\infty}(\Omega)}\leq \varepsilon.
\end{equation*}
Since $\varepsilon$ is an arbitrary constant, letting $\varepsilon\downarrow 0$ yields (6.17). If in addition $f\in W^{1,2}$, $f_m\to f$ in $W^{1,2}$ and $C^{\infty}_{c}$ is dense in $C^{1}_{0}\cap W^{1,2}$. The proof is complete.
\end{proof}

\begin{proof}[Proof of Theorem 1.3]
It follows from (1.7) and (6.16) that     
\begin{equation*}
\sup_{0\leq t\leq T_0}t^{\gamma}\big\|N(v,q)\big\|_{\infty}(t)\leq C\big\|f\big\|_{\infty}^{1-\alpha}\big\|\nabla f\big\|_{\infty}^{\alpha}    \tag{6.18}
\end{equation*}
for all $\tilde{L}^{p}$-solutions for $v_0=\p\partial f$, $f\in C^{\infty}_{c}$ for some $T_0>0$. Since $v=S(t)\p\partial f$ and $S(t)$ is an analytic semigroup on $C_{0,\sigma}$ \cite{AG1}, we are able to extend $T_0$ up to an arbitrary time. Since $C_{c}^{\infty}$ is dense in $C^{1}_{0}\cap W^{1,2}$ by Proposition 6.4, we are able to extend (6.18) for $f\in C^{1}_{0}\cap W^{1,2}$. We proved (1.8). The proof is complete.
 \end{proof}

\vspace{2pt}

\begin{rem}
We used $\tilde{L}^{p}$-theory in order to establish (1.7) since $L^{p}$-theory may not be available for general unbounded domains (see \cite{GHHS} for $L^{p}$-theory for uniformly $C^{3}$-domains). If $L^{p}$-theory is available, the statement of Theorem 1.3 is valid by replacing $\tilde{L}^{p}$ to $L^{p}$.  
\end{rem}

\vspace{2pt}

\begin{proof}[Proof of Theorems 1.2 and 1.1]
Since bounded and exterior domains of class $C^{3}$ are strongly admissible, Theorem 1.2 holds. It remains to show (1.5) for all $t>0$ for bounded domains. It is shown in \cite[Remark 5.4 (i)]{AG1} that the maximum of $S(t)v_0$ for $v_0\in C_{0,\sigma}$ exponentially decays as $t\to \infty$, i.e.,  
\begin{equation*}
\big\|S(t)v_0\big\|_{\infty}\leq Ce^{-\mu t}||v_0||_{\infty}\quad \textrm{for}\ t\geq 0
\end{equation*}
with some constants $\mu>0$ and $C>0$. It follows that
\begin{align*}
\big\|S(t)\p\partial f\big\|_{\infty}
&=\big\|S(t-1)S(1)\p\partial f\big\|_{\infty}\\
&\leq Ce^{-\mu (t-1)}\big\|f\big\|_{\infty}^{1-\alpha}\big\|\nabla f\big\|_{\infty}^{\alpha}\quad \textrm{for}\ t\geq 1.
\end{align*}
Thus, the estimate (1.5) is valid for all $t>0$ for bounded domains. We proved Theorem 1.1.
\end{proof}

\section*{acknowledgements}
The most of this work was done at Nagoya University. The author is grateful to Professor Toshiaki Hishida for valuable comments on this work. The author thanks the referee for helpful remarks and informing him of the paper \cite{Chung} related to Proposition B.1. This work was supported by JSPS through the Grant-in-aid for JSPS Fellow No. 26-2251 and Research Activity Start-up 15H06312 and by Kyoto University Research Funds for Young Scientists (Start-up) FY 2015.

\vspace{15pt}

\appendix

\section{$L^{1}$-type results for the Stokes equations in a half space}

\vspace{5pt}

In Appendix A, we recall an existence result for the dual problem (5.7)--(5.10) on $L^{1}$ and give a proof for Proposition 5.4. 

\vspace{5pt}

\begin{prop}
For $f\in C^{\infty}_{c,\sigma}(\R_{+}\times (0,T))$, there exists a smooth solution $(\varphi,\pi)$ of (5.7)--(5.10) in $\overline{\R_{+}}\times [0,T]$ satisfying $\varphi\in {{S}}$ and $\nabla\pi\in L^{\infty}\big(0,T; L^{1}(\R_{+})\big)$.
\end{prop}

\begin{proof}
The assertion is essentially proved in \cite[Proposition 2.4]{A2}. It is proved that smooth solutions $(\varphi,\pi)$ exist and satisfy $\partial_{t}^{s}\partial_{x}^{k}\varphi, \nabla \pi\in L^{\infty}(0,T; L^{1})$ ($0\leq 2s+|k|\leq 2$), $\varphi=0$ on $\{x_n=0\}\cup \{t=T\}$ and 
\begin{equation*}
\partial_{n}\varphi\in L^{\infty}\Big(0,T; L^{\infty}\big(\mathbb{R}_{+}; L^{1}\big(\mathbb{R}^{n-1}\big)\big)\Big).     
\end{equation*}
Here, $L^{\infty}(\mathbb{R}_{+}; L^{1}(\mathbb{R}^{n-1}))$ denotes the space of all essentially bounded functions $g(\cdot,x_n): \mathbb{R}_{+}\to L^{1}(\mathbb{R}^{n-1})$ and is equipped with the norm $||g||_{L^{\infty}(\mathbb{R}_{+}; L^{1}(\mathbb{R}^{n-1}))}=\textrm{ess sup}_{x_n>0}||g||_{L^{1}(\mathbb{R}^{n-1})}(x_n)$.

The solution $\varphi$ satisfies $x_n^{-1}\varphi\in L^{\infty}(0,T; L^{1})$ (i.e., $\varphi\in S$). In fact, by $\varphi=0$ on $\{x_n=0\}$ and 
\begin{equation*}
\big\|\varphi\big\|_{L^{1}(\mathbb{R}^{n-1})}(x_n)\leq x_n\big\|\partial_{n}\varphi\big\|_{L^{\infty}(\mathbb{R}_{+}; L^{1}(\mathbb{R}^{n-1}))},
\end{equation*}
it follows that 
\begin{equation*}
\big\|x_n^{-1}\varphi\big\|_{L^{1}(\R_{+})}
\leq \int_{0}^{\infty}\frac{1}{x_n}\big\|\varphi\big\|_{L^{1}(\mathbb{R}^{n-1})}(x_n)\dd x_{n}
\leq \big\|\partial_{n}\varphi\big\|_{L^{\infty}(\mathbb{R}_{+}; L^{1}(\mathbb{R}^{n-1}))}+\big\|\varphi\big\|_{L^{1}(\R_{+})}
\end{equation*}
so $x_n^{-1}\varphi\in L^{\infty}\big(0,T; L^{1}\big)$ and $\varphi\in S$.
\end{proof}

\vspace{5pt}

We give a proof for Proposition 5.4. 

\vspace{5pt}

\begin{prop}
Under the assumption of Lemma 5.2, the condition (5.3) is extendable for all $\varphi\in C^{\infty}_{c}(\overline{\R_{+}}\times [0,T])$ satisfying 
$\varphi=0$ on $\{x_n=0\}\cup \{t=T\}$.
\end{prop}

\begin{proof}[Proof of Proposition 5.4]
We cutoff $\varphi\in S$ as $|x|\to\infty$. Let $\theta\in C^{\infty}_{c}[0,\infty)$ be a smooth cutoff function satisfying $\theta\equiv 1$ in $[0,1]$ and $\theta\equiv 0$ in $[2,\infty)$. We set $\theta_{m}(x)=\theta(|x|/m)$ for $m\geq 1$ and $\varphi_{m}=\varphi\theta_{m}$. Then, (5.3) holds for $\varphi_m$ by Proposition A.2, i.e., 
\begin{align*}
0&=\int_{0}^{T}\int_{\R_{+}}(v\cdot (\partial_{t}\varphi_{m}+\Delta \varphi_{m})-\nabla q\cdot \varphi_{m})\dd x\dd t  \\
&=(v,\partial_{t}\varphi_{m})
+(v,\Delta\varphi_{m})
+(-\nabla q, \varphi_{m}).
\end{align*}
Since $\varphi\in S$ satisfies $\partial_{t}^{s}\partial_{x}^{k}\varphi\in L^{\infty}(0,T; L^{1})$ ($0\leq 2s+|k|\leq 2$), the first two terms converge to $(v,\partial_{t}\varphi+\Delta \varphi)$ as $m\to\infty$. Since $\varphi$ satisfies $x_n^{-1}\varphi\in L^{\infty}(0,T; L^{1})$, the last term converges to $(-\nabla q,\varphi)$. Thus, the condition (5.3) is extendable for all $\varphi\in S$.  
\end{proof}

\begin{proof}[Proof of Proposition A.2]
We show that the condition (5.3) is extendable for all $\varphi\in C_{c}^{\infty}(\overline{\R_{+}}\times [0,T))$ satisfying $\varphi=0$ on $\{x_n=0\}$. Let $\theta\in C^{\infty}_{c}[0,\infty)$ be the smooth cut-off function as above and set $\rho_{m}(x_n)=1-\tilde{\theta}_{m}(x_n)$ by $\tilde{\theta}_{m}(x_n)=\theta(m x_n)$ for $m\geq 1$ so that $\rho_{m}\in C^{\infty}[0,\infty)$ satisfies $\rho_{m}\equiv 0$ for $x_n\leq 1/m$ and $\rho_{m}\equiv 1$ for $x_n\geq 2/m$. We substitute $\tilde{\varphi}_{m}=\varphi\rho_{m}$ into (5.3) and observe that 
\begin{align*}
0&=\int_{0}^{T}\int_{\R_{+}}\big(v\cdot(\partial_{t}\tilde{\varphi}_{m}+\Delta \tilde{\varphi}_{m})-\nabla q\cdot \tilde{\varphi}_{m}\big)\textrm{d}x\textrm{d}t  \\
&=(v,\partial_{t}\tilde{\varphi}_{m})
+(v,\Delta\tilde{\varphi}_{m})
+(-\nabla q, \tilde{\varphi}_{m})
\end{align*}
The first term converges to $(v,\partial_{t} \varphi)$. Since $\varphi$ is vanishing on $\{x_n=0\}$, the last term converges to $(-\nabla q,\varphi)$. 

We show that the second term converges to $(v,\Delta \varphi)$. Since  
\begin{align*}
\Delta\tilde{\varphi}_{m}
&=\Delta\varphi\rho_{m}+2\partial_{n}\varphi\partial_{n}\rho_{m}+\varphi\partial^{2}_{n}\rho_{m}\\
&=\Delta\varphi(1-\tilde{\theta}_{m})-2\partial_{n}\varphi\partial_{n}\tilde{\theta}_{m}-\varphi\partial^{2}_{n}\tilde{\theta}_{m},
\end{align*}
it follows that 
\begin{align*}
\int_{0}^{T}\int_{\R_{+}}v\cdot \Delta(\varphi-\tilde{\varphi}_{m})\dd x\dd t
&=\int_{0}^{T}\int_{\R_{+}}v\cdot (\Delta \varphi\tilde{\theta}_{m}+2\partial_{n}\varphi\partial_{n}\tilde{\theta}_{m}+\varphi\partial_{n}^{2}\tilde{\theta}_{m}  )\dd x\dd t  \\
&=:I_{m}+II_{m}+III_{m}.
\end{align*}
The first term $I_m$ converges to zero since $\tilde{\theta}_{m}$ is supported in $\{0\leq x_n\leq 2/m\}$. We show that $II_{m}$ converges to zero. We take $R>0$ such that $\textrm{spt}\ \varphi\subset B_{0}(R)\times [0,T]$ and set 
\begin{equation*}
\eta^{R}_{m}(t)=\sup\left\{\big|v(x,t)\big|\ \Bigg|\ |x|\leq R,\ \frac{1}{m}\leq x_n\leq \frac{2}{m}   \right\}.
\end{equation*}
We observe that $\eta^{R}_{m}(t)\to 0$ as $m\to\infty$ for each $t\in (0,T]$ since $v$ is vanishing on $\{x_n=0\}$. Moreover, $\eta_{m}^{R}(t)$ is estimated by $C/t^{\gamma}$ with some constant $C$ by (5.4). Since $\partial_{n}\tilde{\theta}_{m}$ is supported in $\{1/m\leq |x|\leq  2/m\}$ and $||\partial_{n}\tilde{\theta}_{m}||_{\infty}\leq m||\partial_{n}{\theta}||_{\infty}$, it follows that 
\begin{align*}
\big|II_{m}\big|
&\leq Cm\int_{0}^{T}\eta_{m}^{R}(t)\dd t\int_{\frac{1}{m}}^{\frac{2}{m}}\big\|\partial_{n}\varphi\big\|_{L^{1}(\mathbb{R}^{n-1})}(x_n,t)\dd x_n  \\
&\leq C\Big(\sup\Big\{\big\|\partial_{n}\varphi\big\|_{L^{1}(\mathbb{R}^{n-1})}(x_n,t)\ \Big|\ x_n\in \mathbb{R}_{+}, t\in [0,T] \Big\} \Big)\int_{0}^{T}\eta_{m}^{R}(t)\dd t\to 0\quad \textrm{as}\ m\to\infty.
\end{align*}
It remains to show that $III_m\to 0$. Since $\varphi$ is vanishing on $\{x_n=0\}$, we have 
\begin{align*}
\varphi(x',x_n)=\int_{0}^{x_n}\frac{\partial \varphi}{\partial s}(x',s)\dd s.
\end{align*}
We estimate 
\begin{align*}
\int_{\frac{1}{m}}^{\frac{2}{m}}||\varphi||_{L^{1}(\mathbb{R}^{n-1})}(x_n)\dd x_n
&\leq \int_{\frac{1}{m}}^{\frac{2}{m}}\int_{0}^{x_n}||\partial_n\varphi||_{L^{1}(\mathbb{R}^{n-1})}(s)\dd s \dd x_n\\
&\leq \frac{3}{2m^{2}}||\partial_n\varphi||_{L^{\infty}(\mathbb{R}_{+}; L^{1}(\mathbb{R}^{n-1}))},
\end{align*}
where the time variable is suppressed. Since $||\partial_n^{2}\tilde{\theta}_m||_{\infty}\leq m^{2}||\partial_n^{2}{\theta}||_{\infty}$, it follows that 
\begin{align*}
|III_m|
&\leq Cm^{2}\int_{0}^{T}\eta_m^{R}(t)\dd t\int_{\frac{1}{m}}^{\frac{2}{m}}||\varphi||_{L^{1}(\mathbb{R}^{n-1})}(x_n,t)\dd x_n\\
&\leq C'\Bigg(\sup_{t\in [0,T]}||\partial_n\varphi||_{L^{\infty}(\mathbb{R}_{+}; L^{1}(\mathbb{R}^{n-1}))}(t)\Bigg)\int_{0}^{T}\eta_{m}^{R}(t)\dd t\to 0\quad \textrm{as}\ m\to \infty.
\end{align*}
We proved that the condition (5.3) is extendable for all $\varphi\in C^{\infty}_{c}\big(\overline{\R_{+}}\times [0,T)\big)$ satisfying $\varphi=0$ on $\{x_n=0\}$. By a similar cut-off argument near $t=T$, we are able to extend (5.3) for all $\varphi\in C^{\infty}_{c}\big(\overline{\R_{+}}\times [0,T]\big)$ satisfying $\varphi=0$ on $\{x_n=0\}\  \cup \{t=T\}$. The proof is now complete.  
\end{proof}

\vspace{10pt}

\section{Uniqueness of the heat equation}

In Appendix B, we give some uniqueness results for the heat equation in the whole space and a half space, used in the proof of Lemma 5.2 and Proposition 6.2. The uniqueness of the heat equation is studied under very weak regularity conditions near time zero (see, e.g., \cite{Sha}, \cite{ChungKim}, \cite{Chung}). We prove uniqueness under a bound for $t^{\gamma}||v||_{\infty}$ and $\gamma<1$ based on a duality argument.

\vspace{5pt}

\begin{prop}
Let $T>0$ and $n\geq 1$. Let $v\in L^{1}_{\textrm{loc}}(\mathbb{R}^{n}\times [0,T))$ satisfy
\begin{align*}
\int_{0}^{T}\int_{\mathbb{R}^{n}}v(\partial_t \varphi+\Delta \varphi)\dd x\dd t=0   \tag{B.1}
\end{align*}
for all $\varphi\in C_{c}^{\infty}(\R\times [0,T))$. Assume that 
\begin{align*}
\sup_{0<t\leq T}t^{\gamma}||v||_{L^{\infty}(\R)}<\infty  \tag{B.2}
\end{align*}
for some $\gamma\in [0,1)$. Then, $v\equiv 0$.
\end{prop}

\begin{proof}
Under the assumption (B.2), the condition (B.1) is extendable for all $\varphi\in C^{\infty}(\R\times [0,T])$ satisfying 
\begin{align*}
\partial_t^{s}\partial_x^{k}\varphi\in L^{\infty}(0,T; L^{1}(\R))\quad \textrm{and}\quad \varphi=0\quad \textrm{on}\ \{t=T\},  \tag{B.3}
\end{align*}
for $2s+|k|\leq 2$. In fact, for $\varphi\in C^{\infty}(\R\times [0,T])$ satisfying $\partial_t^{s}\partial_x^{k}\varphi\in L^{\infty}(0,T; L^{1})$ and $\varphi\in C^{\infty}_{c}([0,T); L^{1})$, we set $\varphi_m=\varphi\theta_m$ by the cut-off function $\theta_{m}(x)=\theta(|x|/m)$ and $\theta\in C_{c}^{\infty}[0,\infty)$ satisfying $\theta= 1$ in $[0,1]$ and $\theta=0$ in $[2,\infty)$. We substitute $\varphi_m$ into (B.1) to get 
\begin{align*}
0=\int_{0}^{T}\int_{\mathbb{R}^{n}}v\Big((\partial_t \varphi+\Delta \varphi)\theta_m+2\nabla \varphi\cdot \nabla \theta_m+\varphi\Delta \theta_m\Big)\dd x\dd t.
\end{align*}
Since $v$ is integrable near time zero by (B.2), the first term converges to $(v,\partial_t\varphi+\Delta \varphi)$. The other terms vanish as $m\to\infty$ since $\partial_x^{k}\varphi\in L^{\infty}(0,T; L^{1})$ and $||\partial_x^{k}\theta_m||_{\infty}\leq C/m^{|k|}$. By a similar cut-off argument near $t=T$, we are able to extend (B.1) for all $\varphi$ satisfying (B.3).

For an arbitrary $f\in C_{c}^{\infty}(\mathbb{R}^{n}\times (0,T))$, we set $\tilde{f}(\cdot, t)=-{f}(\cdot, T-t)$ and 
\begin{align*}
\tilde{\varphi}(\cdot,t)=\int_{0}^{t}e^{(t-s)\Delta}\tilde{f}(s)\dd s.
\end{align*}
Then, by $L^{1}$-estimates of the heat semigroup, we have $\partial_t^{s}\partial_x^{k}\tilde{\varphi}\in L^{\infty}(0,T; L^{1})$ for $2s+|k|\leq 2$. Moreover, $\tilde{\varphi}$ satisfies $\partial_t \tilde{\varphi}-\Delta \tilde{\varphi}=\tilde{f}$ in $\mathbb{R}^{n}\times (0,T)$ and $\tilde{\varphi}=0$ on $\{t=0\}$. We set $\varphi(\cdot,t)=\tilde{\varphi}(\cdot,T-t)$ and obtain $\varphi$ satisfying $\partial_t {\varphi}+\Delta {\varphi}={f}$ in $\mathbb{R}^{n}\times (0,T)$ and ${\varphi}=0$ on $\{t=T\}$. Since $\varphi$ satisfies (B.3), it follows that 
\begin{align*}
\int_{0}^{T}\int_{\R}vf\dd x\dd t
=\int_{0}^{T}\int_{\R}v(\partial_t\varphi+\Delta\varphi)\dd x\dd t
=0.
\end{align*}
We proved $v\equiv 0$. 
\end{proof}

\vspace{5pt}

\begin{lem}
Let $v\in L^{1}_{\textrm{loc}}(\mathbb{R}^{n}_{+}\times [0,T))$ satisfy
\begin{align*}
\int_{0}^{T}\int_{\mathbb{R}^{n}_{+}}v(\partial_t \varphi+\Delta \varphi)\dd x\dd t=0   \tag{B.4}
\end{align*}
for all $\varphi\in C_{c}^{\infty}(\mathbb{R}^{n}_{+}\times [0,T))$. Assume that $v(\cdot,t)\in C(\overline{\mathbb{R}^{n}_{+}})$ satisfies $v=0$ on $\partial\mathbb{R}^{n}_{+}$ and 
\begin{align*}
\sup_{0<t\leq T}t^{\gamma}||v||_{L^{\infty}(\mathbb{R}^{n}_{+})}<\infty  \tag{B.5}
\end{align*}
for some $\gamma\in [0,1)$. Then, $v\equiv 0$.
\end{lem}

\begin{proof}
We reduce the problem by reflection. By (B.5) and the Dirichlet boundary condition, we are able to extend (B.4) for all $\varphi\in C_{c}^{\infty}(\overline{\mathbb{R}^{n}_{+}}\times [0,T))$ satisfying $\varphi=0$ on $\{x_n=0\}$ as we did in the proof of Proposition A.2. We set the odd extension of $v$ by 
\begin{align*}
\tilde{v}(x',x_n)=
\begin{cases}
&v(x',x_n)\qquad\ \textrm{for}\ x_n\geq 0,\\
&-v(x',-x_n)\quad\textrm{for}\ x_n<0.
\end{cases}
\end{align*}
We show that $\tilde{v}$ satisfies (B.1). For an arbitrary $\phi\in C_{c}^{\infty}(\mathbb{R}^{n}\times [0,T))$, we set $\varphi(\cdot,x_n)=\phi(\cdot,x_n)-\phi(\cdot,-x_n)$ for $x_n\geq 0$. It follows that 
\begin{align*}
\int_{0}^{T}\int_{\mathbb{R}^{n}}\tilde{v}(\partial_t \phi+\Delta \phi)\dd x\dd t
&=\int_{0}^{T}\int_{\{x_n\geq0\}}v(x',x_n,t)\Big(\partial_t \phi(x',x_n,t)+\Delta \phi(x',x_n,t)\Big)\dd x\dd t\\
&-\int_{0}^{T}\int_{\{x_n<0\}}v(x',-x_n,t)\Big(\partial_t \phi(x',x_n,t)+\Delta \phi(x',x_n,t)\Big)\dd x\dd t\\
&=\int_{0}^{T}\int_{\mathbb{R}^{n}_{+}}v(\partial_t \varphi+\Delta \varphi)\dd x\dd t.
\end{align*}
Since $\varphi\in C^{\infty}_{c}(\overline{\mathbb{R}^{n}_{+}}\times [0,T))$ satisfies $\varphi=0$ on $\{x_n=0\}$, the right-hand side equals zero by (B.4). By Proposition B.1, $v\equiv 0$ follows. The proof is complete.
\end{proof}

\vspace{10pt}


\end{document}